\documentclass[a4paper,11pt]{amsart}
\usepackage{amsmath}
\usepackage{amssymb}
\usepackage{amsfonts}
\usepackage{mathrsfs}
\usepackage{graphicx}
\usepackage{epsfig}
\usepackage{caption,subcaption}
\usepackage{hyperref} 
\usepackage{enumitem}
\usepackage[T1]{fontenc}
\usepackage{enumitem}
\usepackage[left=2.5cm,right=2.5cm,top=2.5cm,bottom=2.5cm]{geometry}
\numberwithin{equation}{section}       
\usepackage{graphics}
\usepackage{color}
 \theoremstyle{plain}
\newtheorem{theorem}{Theorem}[section]
\newtheorem*{mtheorem}{Main Theorem}

\newtheorem*{iitheorem}{Informal Statement of Theorem \ref{thm:curves}}
\newtheorem*{atheorem}{Theorem A.1}
\newtheorem{prop}{Proposition}[section]

\newtheorem{coro}[prop]{Corollary}
\newtheorem{lemma}[prop]{Lemma}
\newtheorem{fact}[prop]{Fact}
\newtheorem{example}[prop]{Example}

\newtheorem*{mainlemma}{Main Lemma}

\usepackage{hyperref}
\newcommand{\ARNOLD}[1]{}  

\newtheorem{definition}[prop]{Definition}

\newtheorem{remark}[prop]{Remark}
\newtheorem{exam}[prop]{Example}

\newtheoremstyle{citing}
  {3pt}
  {3pt}
  {\itshape}
  {}
  {\bfseries}
  {.}
  {.5em}
  {\thmnote{#3}}

\theoremstyle{citing}

\newcommand{\C}{\mathbb{C}}
\newcommand{\D}{\mathbb{D}}

\newcommand{\R}{\mathbb{R}}

\newcommand{\Z}{\mathbb{Z}}

\newcommand{\sign}{\mbox{sign\,\,}}

\newcommand{\teta}{\widetilde{\teta}}

\newcommand{\eps}{\varepsilon}

\DeclareMathOperator{\diam}{diam}

\begin{document}
\title[]{Positive Transversality via 
transfer operators and holomorphic motions with applications
to monotonicity for interval maps}
\author{Genadi Levin, Weixiao Shen and Sebastian van Strien}

\date{24 Jan 2019}

\begin{abstract}
In this paper we will develop a general approach which shows that generalized {\lq\lq}critical relations{\rq\rq}
of  families of locally defined holomorphic maps on the complex plane unfold transversally.
The main idea is to define a transfer operator,  which is a local analogue of the
Thurston pullback operator, using holomorphic motions.
Assuming a so-called lifting property is satisfied, we obtain information about the spectrum of this transfer operator
and thus about transversality.
An important new feature of our method is that it is not global:  the maps we consider are only required to be defined and holomorphic on
a neighbourhood of some finite set.

We will  illustrate this method by obtaining
 transversality for a wide class of one-parameter families of interval and circle maps,
for example for maps with flat critical points,   
but also for maps
with  complex analytic extensions such as certain polynomial-like maps.
As in Tsujii's approach \cite{Tsu0,Tsu1}, for real maps we  obtain {\em positive} transversality (where $>0$ holds instead of just $\ne 0$),
and thus  monotonicity of entropy for these families, and also (as an easy application) for
the real quadratic family.

This method additionally  gives results for unimodal families of the form $x\mapsto |x|^\ell+c$
for $\ell>1$ not necessarily an even integer and $c$ real.
\end{abstract}

\maketitle
\setcounter{tocdepth}{1}
\setcounter{secnumdepth}{2}

\tableofcontents

\section{Introduction}

This paper is about bifurcations in families of (real and complex) one dimensional dynamical systems.\footnote{This paper  is based on the preprint
{\em Monotonicity of entropy and positively oriented transversality for families of interval maps},
see https://arxiv.org/abs/1611.10056}  For example, 
for real one-dimensional dynamical systems, we have a precise combinatorial description on the dynamics in terms of the so-called kneading sequences. One simple but very important question is how the kneading sequence varies in families of such systems. For the real quadratic family $f_a(x) = x^2 + a$, it is known that the kneading sequence depends monotonically  on the parameter $a$ (with respect to the natural order defined for kneading sequences). Interestingly the proofs of this result, by Milnor-Thurston, Douady-Hubbard and Sullivan, make use of Teichm\"uller theory,  uniqueness in Thurston's realisation theorem,  or quasiconformal rigidity theory and that the map $f_a$ is quadratic, see   
\cite{MT, Su} and also \cite{D}.

To answer the above monotonicity question it is enough to show that when $f^n_{a'}(0)=0$  for some $a'\in \R$, 
then there exists no other parameter $a''\in \R$ for which $f^n_{a''}(0)=0$
and for which $f_{a'}$ and $f_{a''}$ also have the same (periodic) kneading sequence. 

If $\frac{d}{d a} f_a^n(0)\left.\right|_{a=a'}\ne 0$ (which is called {\em transversality})
then one has local uniqueness in the following sense:
there exists $\epsilon>0$ so that the  kneading invariant of $f_a$ for $a\in (a'-\epsilon,a')$, for $a\in (a',a'+\epsilon)$
and for $a=a'$ are all different. 
%
%
It turns out that {\em global} uniqueness and monotonicity follows from 
$$Q:= \frac{\frac{d}{d a} f_a^n(0)\left.\right|_{a=a'}}{Df_{a'}^{n}(f_{a'}(0))}>0  $$
(which we call {\em positive transversality}).   Tsujii  gave an alternative proof of the above monotonicity for the quadratic family by showing
that this inequality holds \cite{Tsu0,Tsu1}.

For general (complex) holomorphic families of maps with several critical points which all are eventually periodic
there exists a similar expression $Q$. Again $Q\ne 0$ implies that the bifurcations
are non-degenerate and hence the corresponding critical relations unfold transversally. 
In this paper we show that the inequality $Q\neq 0$ holds provided the spectrum of some operator $\mathcal{A}$ does not contain 1, and that $Q > 0$ holds if additionally the spectrum of  $\mathcal{A}$ is contained in the closed unit disc and  the family of maps is
real.

We define this operator $\mathcal A$ by considering how the speed of a  (holomorphic) motion of the orbits of critical points is lifted under the dynamics.  The novelty of our method, described in Proposition 5.1 and Theorem 6.1, is to show if these holomorphic motions have the lifting property, i.e. can  be lifted infinitely many times over the same domain, then the operator $\mathcal A$ has the above spectral properties.

It turns out that the lifting property makes minimal use of the global dynamics of the holomorphic extension of the dynamics. Thus we can obtain transversality properties of families $f_t$ of maps defined on open subsets of the
complex plane so that  $f_0$ has a finite invariant marked set, e.g., $f_0$ is {\lq}critically finite{\rq}.

The methods developed in this paper give a new and simple proof of well-known results
for families of polynomial maps, rational or entire maps,
but also applies to many other families for which no techniques were available.
For example we obtain monotonicity for the family of maps $$f_c(x)= b e^{-1/|x|^\ell}+c$$
having a flat critical point at $0$. We also obtain partial monotonicity 
for the family $$f_c(x)=|x|^\ell+c$$ when $\ell$ is large. 

As mentioned, the aim of this paper is to deal with families of maps which are only locally holomorphic.
This means that the approach pioneered by Thurston, and developed by Douady and Hubbard in \cite{DH2}, cannot be applied.
In Thurston's approach, when $f$ is a globally defined holomorphic map, $P$ is a finite $f$-forward invariant set containing
the postcritical set and the Thurston map 
$\sigma_f\colon
\mbox{Teich}(\widehat{\C}\setminus P) \to \mbox{Teich}(\widehat{\C}\setminus P)$
 is defined by pulling back an almost holomorphic structure.
It turns out that $\sigma_f$ is contracting, see \cite[Corollary 3.4]{DH2}.
In  Thurston's result on the topological realisation of rational maps,  Douady \& Hubbard \cite{DH2}
use that the dual of the derivative of the Thurston map $\sigma_f$ is equal to the Thurston pushforward operator $f_*$.

The pushforward operator acts on space of quadratic differentials, and the
 injectivity of  $f_*-id$ has been used to obtain transversality results
for rational maps or even more general spaces of maps,  see Tsujii \cite{Tsu0,Tsu1}, Epstein \cite{Ep2}, Levin \cite{Le1,Le2,Le3}, Makienko
\cite{Mak}, Buff \& Epstein \cite{BE} and Astorg \cite{Ast}.  See also earlier \cite{Le4}, \cite[Proposition  22\footnote{This reference was provided by the author.}]{Ep1}, \cite{Ep3}.
For a short and elementary proof that   critical relations unfold transversally
in the setting of rational maps,  see \cite{LSvS}. For the spectrum of $f_*$ see also   \cite{LSY0}, \cite{LSY}, \cite{Le0}, \cite{ELS} and  \cite{Betal}.   In \cite{Str}, transversality is shown for rational maps 
for which each critical point is  mapped into a hyperbolic set, using the  uniqueness part of
Thurston's realisation result.

%

However if for example  $f\colon U\to V$ is a polynomial-like map then each point in the boundary of $V$ is a singular
value,  $\mbox{Teich}(V\setminus P)$ is
 infinite dimensional, Thurston's algorithm is only locally well-defined and it is not clear whether it is locally contracting.

The purpose of our paper is to bypass this issue, by going back to the original Milnor-Thurston approach.
Milnor and Thurston \cite{MT} associated to the space of quadratic maps and
the combinatorial information of a periodic orbit,
a map
 which assigns to a $q$ tuple of points a new $q$ tuple of points,
$$F: (z_1,...z_q)\mapsto (\hat z_1,...,\hat z_q)$$
where $\hat z_q=0$ and
$$f_{z_1}(\hat z_i)=z_{i+1 \!\!\!\!  \mod q}$$
where $f_{c}(z)\equiv z^2+c$.
Since $F$ is many-valued, Milnor \& Thurston considered a lift $\tilde{F}$ of this map to the universal cover and apply Teichm\"uller theory to show that $\tilde{F}$ is strictly contracting in the Teichmuller metric of the universal cover.




We bypass this issue by rephrasing their approach locally (via holomorphic motions). This is done in the set-up of so-called marked maps (and their local deformations) which include particularly critically finite maps with any number of critical (singular) points.
In the first part of the paper we prove general results, notably the Main Theorem, which show that under the assumption that
some lifting property holds for the deformation,
either some critical relation persists along some non-trivial manifold in parameter space  or one has
transversality, i.e. the critical relation unfolds transversally. Here the lifting property is an assumption that sequences of successive lifts of holomorphic motions
are compact. In the second part of the paper, we then  show that this lifting property holds not only in
previously considered global cases but also for interesting classes of maps where the {\lq}pushforward{\rq} approach breaks down.

More precisely, we define a transfer operator $\mathcal A$ by its action on infinitesimal holomorphic motions on $P$.
It turns out that 
if the lifting property holds, then the spectrum of the operator $\mathcal A$ lies inside the unit disc.
Moreover, if the operator $\mathcal A$ has no eigenvalue $1$ then  transversality holds; in the real case
one has even positive transversality (the sign of some determinant is positive).  One of the main
steps in the proof of the Main Theorem is then to show that if the operator has an eigenvalue $1$ then
the critical relation persists along a non-trivial manifold in parameter space.
It turns out that for
globally defined critically finite maps $f$  the transfer operator $\mathcal A$ can be identified with (dual of) $f_*$.

By verifying the lifting property we recover previous results such as transversality for rational maps, but also
obtain transversality for many interesting families of polynomial-like mappings and families of maps with essential singularities.
For real local maps our approach gives the {\lq}positive transversality{\rq} condition which first appeared in \cite{Tsu0,Tsu1} and therefore
monotonicity of entropy for certain families of real maps.

\section{Statement of Results}

\subsection{The notion of a  marked map}\label{subsec:markedmap}
\def\MM{\mathcal M}
A {\em marked map} is a map $g$ from the union of a finite set $P_0$ and an open set $U\subset \C$ mapping
into $\C$ such that
\begin{itemize}
\item there exists a finite set $P\supset P_0$ such that $g(P)\subset P$ and $P\setminus P_0\subset U$;
\item $g|U$ is holomorphic and $g'(x)\not=0$ for $x\in P\setminus P_0$.
\end{itemize}
Let $\{c_{0,1}, \dots,c_{0,\nu}\}$ denote the (distinct) points in $P_0$ and write
${\bf c}_0={\bf c}_0(g)=(c_{0,1},\dots ,c_{0,\nu})$
and
$${\bf c}_1={\bf c}_1(g)=(g(c_{0,1}),\dots, g(c_{0,\nu})):=(c_{1,1},c_{1,2},\ldots, c_{1,\nu}).$$

\begin{remark}
So $P$ is a forward invariant set for $g$, and $g$ is only  required to be holomorphic (and  defined)  on a
neighbourhood of
$P\setminus P_0$.  A  marked map $g$ does not need to be defined in a neighbourhood of $P_0$.
 In applications, points in $P_0$ will be where   some extension of $g$ has a critical point, an (essential) singularity or even
 where $g$ has a discontinuity. In this sense marked maps correspond to a generalisation of
 the notion of critically finite maps. \end{remark}

\subsection{Holomorphic deformations}\label{subsec:holdefor}
A {\em local holomorphic deformation} of $g$
is a triple $(g,G, \textbf{p})_W$ with the following properties:
\begin{enumerate}
\item $W$ is an open connected subset of $\C^\nu$ containing  $\textbf{c}_1(g)$;
\item $\textbf{p}=(p_1,p_2,\ldots, p_\nu):W\to \C^\nu$ is a holomorphic map,
so that $\textbf{p}(\textbf{c}_1)={\bf c}_0(g)$ (and so all coordinates of $\textbf{p}(\textbf{c}_1)$ are distinct).
\item $G: (w,z)\in W\times U \mapsto G_w(z)\in  \C$ is a holomorphic map such that $G_{\textbf{c}_1}=g|U$.
\end{enumerate}

\begin{example}  \label{ex1} 
(i) The simplest local holomorphic deformation of $g$ is of course the trivial one: $G_w(z)=g(z)$, $p(w)=c_0$, $\forall w$.  \\
(ii) If $g$ is a marked map with $\nu=1$, i.e., $P_0=\{c_0\}$, then $G_w(z)=g(z)+(w-g(c_0))$, $p(w)\equiv c_0$, defines a local holomorphic deformation of $g$.  \\
(iii) If $g$ is rational map with $\nu$ critical points $c_1,\dots,c_\nu$  with multiplicity $\mu_1,\dots,\mu_\nu$
then there exists a local holomorphic deformation $(g,G,\textbf{p})_W$ of $g$ so that $w_i=G_w(\textbf{p}_i(w))$ is a critical value of $G_w$ for each  $w=(w_1,\dots,w_\nu)\in W$ and each $i=1,\dots,\nu$, see Appendix C. 
\end{example} 

\subsection{Transversal unfolding of critical relations}
Let us fix $(g,G, \textbf{p})_W$ as above.
Since $g(P)\subset P$ and $P$ is a finite set,
for each $j=1,2,\ldots, \nu$, exactly one of the following {\em critical relations} holds:
\begin{enumerate}
\item[(a)]  There exists an integer $q_j>0$ and $\mu(j)\in \{1,2,\ldots,\nu\}$ such that
$g^{q_j}(c_{0,j})=c_{0,\mu(j)}$ and $g^k(c_{0,j})\not\in P_0$ for each $1\le k<q_j$;
\item[(b)]\label{ljqj} There exist integers $1\le l_j<q_j$ such that
$g^{q_j}(c_{0,j})=g^{l_j}(c_{0,j})$ and $g^k(c_{0,j})\not\in P_0$ for all $1\le k\le q_j$.
\end{enumerate}
Relabelling these points $c_{0,j}$, we assume that there is $r$ such that
the first alternative happens for all $1\le j\le r$ and the second alternative happens for $r<j\le \nu$.

Define the map
$$\mathcal{R}=(R_1, R_2, \dots, R_\nu)$$
from a neighbourhood of $\textbf{c}_1\in \C^\nu$ into $\C^\nu$ as follows:
for $1\le j\le r$,
\begin{equation}R_j(w)=G_{w}^{q_j-1}(w_j)-p_{\mu(j)}(w)\label{R1-r'} \end{equation}
and for $r<j\le \nu$,
\begin{equation}R_j(w)=G_{w}^{q_j-1}(w_j)-G_{w}^{l_j-1}(w_j),\label{Rr-nu'}\end{equation}
where $w=(w_j)_{j=1}^\nu$.

\begin{definition}\label{def:unfoldtrans}
We say that the holomorphic deformation $(g,G, \textbf{p})_W$ of $g$ satisfies the {\em transversality property}, if the Jacobian matrix
$D\mathcal{R}(\textbf{c}_1)$ is invertible.
\end{definition}

\begin{example}  \label{ex2} 
(i) Assume  that $(g,G,\textbf{p})_W$ is a local holomorphic deformation so that for each $w=(w_1,\dots,w_\nu)\in W$, 
the critical values of $G_w$ are  $w_1,\dots,w_\nu$  
and  $\textbf{p}_1(w),\dots,\textbf{p}_\nu(w)$ are the critical points of $G_w$.
Then $G_{w}(\textbf{p}_i(w))=w_i$. Hence $R_j\equiv 0$ in (\ref{R1-r'}) or (\ref{Rr-nu'}) correspond to 
$$G_{w}^{q_j}(\textbf{p}_j(w)) - p_{\mu(j)}(w) \equiv 0 \mbox{ resp. }G_{w}^{q_j}(\textbf{p}_j(w))-G_{w}^{l_j}(\textbf{p}_i(w)) \equiv 0.$$ 
These equations define the set of parameters $w$
for which the corresponding {\lq}critical relation{\rq} is satisfied within the family $G_w$. \\
(ii) In  Example~\ref{ex3}(ii) we will consider a  holomorphic deformation  $(g,G,\textbf{p})_W$ 
of a  map $g$ so that $G_w$ is  not defined (as an analytic map) in $\textbf{p}_j(w)$, but nevertheless 
the above interpretation is valid.   \\
(iii)
 If we take the trivial deformation $G_w(z)=g(z)$, ${\bf p}(w)={\bf c}_0$, $\forall w\in W$,   then definitions  (\ref{R1-r'}) or (\ref{Rr-nu'})  take the form
 $R_j(w)=g^{q_j-1}(w_j)-c_{0,j'}$ respectively $R_j(w)=g^{q_j-1}(w_j)- g^{l_j-1}(w_j)$. 
Then $D\mathcal{R}(c_1)$ is a diagonal matrix with entries $Dg^{q_j-1}(w_j)$ for $1\le j\le r $
and $Dg^{l_j-1}(w_j) [ Dg^{q_j-l_j}(g^{l_j-1}(w_j))-1]$ for $ r<j\le \nu$.
So the matrix $D\mathcal{R}(c_1)$ is non-degenerate iff  $Dg^{q_j-l_j}(g^{l_j-1}(w_j))\neq 1$ for $r<j\le \nu$.
 It follows immediately that the  
 holomorphic deformation $(g,G,{\bf p})_W$ satisfies the transversality property if and only if $Dg^{q_j-l_j}(g^{l_j}(c_{0,j}))\neq 1$ for $r<j\le\nu$.
We should emphasise that in this setting the condition $\mathcal{R}(w)=0$ has nothing to do with
the presence of critical relations.
\end{example}


\subsection{Real marked maps and positive transversality}
A marked map $g$ is called {\em real} if $P\subset \R$ and for any $z\in U$ we have $\overline{z}\in U$ and $\overline{g(z)}=g(\overline{z})$.  Similarly, a local holomorphic deformation $(g, G,\textbf{p})_W$ of a real marked map $g$ is called {\em real} if
for any $w=(w_1, w_2, \ldots, w_\nu)\in W$, $z\in U$ and $j=1,2,\ldots,\nu$, we have
$\overline{w}= (\overline{w_1}, \overline{w_2},\ldots, \overline{w_\nu})\in W$,
$G_{\overline{w}}(\overline{z})=\overline{G_w(z)}, \text{ and } p_j(\overline{w})=\overline{p_j(w)}$.

\begin{definition}
Let $(g, G, \textbf{p})_W$ be a real local holomorphic deformation of a real marked map $g$.
We say that  the unfolding $(g,G, \textbf{p})_W$ satisfies {\em the `positively oriented' transversality property} if
\begin{equation} Q:= \frac{\det (D\mathcal{R}(\textbf{c}_1))}{\prod_{j=1}^\nu Dg^{q_j-1}(c_{1,j})}>0.\label{eq:trans2}
\end{equation}
\end{definition}
The sign in the previous inequality  means that the intersection of the analytic sets $R_j=0$
$j=1,\dots,\nu$,  
is not only in general position (i.e. {\lq}transversal{\rq}), 
but that the intersection pattern is {\em everywhere  {\lq}positively oriented{\rq}}. 

\subsection{The lifting property}\label{subsec:liftproperty}

Let $X\subset \C$ and  $\Lambda$ be a domain in $\C$ which contains $0$. 
As usual, we say that $h_\lambda$ is a {\em holomorphic motion of $X$ over   $(\Lambda, 0)$}, if 
 $h_\lambda \colon X\to \C$ satisfies:  \\ 
 (i)  $h_0(x)=x$, for all  $x\in X$,  \\
 (ii)  $h_\lambda(x)\ne h_\lambda(y)$ whenever $x\ne y$ and $\lambda\in \Lambda$, $x,y\in X$ and \\
 (iii) $\Lambda \ni \lambda \mapsto h_\lambda(z)$ is holomorphic.   

Given a holomorphic motion $h_\lambda$ of $g(P)$ over $(\Lambda, 0)$, where $\Lambda$ is a domain in $\C$ which contains $0$,
we say that $\hat{h}_\lambda$ is a {\em lift}  of $h_{\lambda}$ {\em over} $\Lambda_0\subset \Lambda$ (with $0\in \Lambda_0$)
with respect to  $(g,G, \textbf{p})_W$ if for all $\lambda\in \Lambda_0$,
\begin{itemize}
\item $\hat{h}_\lambda(c_{0,j})=p_j({\textbf{c}_1}(\lambda))$ for each $j=1,2, \dots, \nu$, with $c_{0,j}\in g(P)$,
where
$$\textbf{c}_1(\lambda)=(h_{\lambda}(c_{1,1}), h_{\lambda}(c_{1,2}),\dots, h_{\lambda}(c_{1,\nu}));$$
\item $G_{\textbf{c}_1(\lambda)}(\hat{h}_{\lambda}(x))=h_{\lambda}(g(x))$ for each $x\in g(P)\setminus P_0$.
\end{itemize}
Clearly such a lift exists, provided $\Lambda_0$ is contained in a sufficiently small neighbourhood of $0$.

We say that the triple $(g,G,\textbf{p})_W$ has the {\em lifting property} if for each holomorphic motion $h_\lambda^{(0)}$ of $g(P)$ over $(\D,0)$
there exist $\eps>0$ and a sequence of holomorphic motions $h_\lambda^{(k)}$, $k=1,2,,\ldots$
of $g(P)$ over $(\D_{\eps}, 0)$ such that for each $k\ge 0$,
\begin{itemize}
\item $h^{(k+1)}_\lambda$ is a lift of $h^{(k)}_\lambda$ over $(\D_{\eps}, 0)$;
%
\item there exists $M>0$ such that $|h_\lambda^{(k)}(x)|\le M$ for all
$x\in g(P)$ and all $\lambda\in \D_{\eps}$.
\end{itemize}
In the case $(g, G, \textbf{p})_W$ is real, we say it has {\em the real lifting property} 
if the corresponding property holds for any real-symmetric  holomorphic motions $h_\lambda^{(0)}$.

\subsection{Statement of the Main Theorem}

\begin{mtheorem}
%
Assume that g does not have a parabolic periodic point in $P\setminus P_0$ and that $(g,G,\textbf{p})_W$ satisfies the lifting property. Then exactly one of the following holds:
\begin{enumerate}
\item[(1)] the holomorphic deformation $(g,G, \textbf{p})_W$ of $g$  satisfies the tranversality property;
\item[(2)] there exists a neighborhood $W'$ of ${\bf c}_1$ such that $\{w\in W': \mathcal{R}(w)=0\}$ is a smooth complex manifold of positive dimension.
\end{enumerate}
Moreover, if $(g,G, \textbf{p})_W$ is real and satisfies the real lifting property then in (1) `the transversality property' can be replaced by `the {\lq}positively oriented{\rq} transversality property'.
\end{mtheorem}

The statement of this theorem is a combination of the more detailed statements in
Theorems~\ref{thm:specAtrans} and \ref{thm:1eigen}.

\subsection{Classical settings where the lifting property holds}
In many cases it is easy to check that the lifting property holds, and therefore the previous theorem
applies.  Indeed, it is easy to see that this holds in the setting
of polynomial or rational maps, see Section~\ref{sec:liftingclassical}.

\subsection{Transversality for new families of maps corresponding to classes $\mathcal F$, $\mathcal E$, $\mathcal E_o$} \label{subsec:defclassesEF} 
In this subsection we will discuss two new settings where the current approach
can be applied to obtain transversality.

Let us first consider families of maps $f_c(z)=f(z)+c$. Here  $f$ is contained in the space $\mathcal F$
of holomorphic maps $f\colon  U\to V$ where
\begin{enumerate}
\item[(a)] $U$ is a bounded open set in $\C$ such that 
$0\in \overline{U}$;
\item[(b)] $V$ is a bounded open set in $\C$; 
\item[(c)] $f: U\setminus \{0\}\to V\setminus \{0\}$ is an unbranched covering; 
\item[(d)] $V\supset B(0;\diam(U)) \supset U$.
\end{enumerate}
Examples of such families are
\begin{example}\label{ex3} 
\begin{enumerate}
\item[(i)] $f_c(z)=z^d+c$, where $U,V$ are suitably large balls and $c\in U$.
\item[(ii)]  $f_c(x)= b e^{-1/|x|^\ell}+c$ where
$\ell\ge 1$, $b> 2(e\ell)^{1/\ell}$, $c\in U$.
Here $U=U^-\cup U^+$, $U^+=-U^-$ are topological disks symmetric w.r.t. the real axis
and $V$ is a punctured disc. That $f_0 \in \mathcal F$ is proved in Corollary~\ref{coro83}.
\end{enumerate}\end{example} 

\begin{theorem} Let $f\in \mathcal{F}$, let $g=f_{c_1}$ and for each $w\in W:=\C$ define $G_w(z)=g(z)+(w-c_1)$ and
${\bf p}(w)=0$ where we assume that $c_1\in U$. Moreover, assume that there exists $q$ so that
$c_n=g^{n-1}(c_1)\in U$ for all $n<q$ and either
$c_q=0$ or $c_q\in \{c_1,\dots,c_{q-1}\}$. Moreover, assume
$c_n\notin \{0,c_1,\dots,c_{n-1}\}$ for $0<n<q$.
Then
\begin{itemize}
\item $(g,G,{\bf p})_W$  satisfies the lifting property and transversality holds.
\item if  $(g,G,{\bf p})_W$ is real, then positive transversality holds.
\end{itemize}
\end{theorem}

Our methods also apply to families of the form $f_w(z)=wf(z)$ where $f$ is contained in
the spaces  $\mathcal{E}$ and $\mathcal{E}_o$ defined as follows.
Consider holomorphic maps $f: D\to V$ such that:
\begin{enumerate}
\item[(a)] $D,V$ are open sets which are symmetric w.r.t. the real line so that
$f(D)=V$  
\item[(b)] Let $I=D\cap \R$ then there exists $c>0$ so that $I\cup \{c\}$ is a (finite or infinite) open interval and $0\in \overline I$, $c\in \mbox{int}(\overline{I})$.
Moreover,   $f$ extends continuously to $\overline{I}$,  $f(I)\subset   \R$ and $\lim_{z\in D, z\to 0} f(z)=0$.
\item[(c)]   Let $D_+$ be the component of $D$ which contains $I\cap (c,\infty)$, where $D_+$ might be equal to $D$.
Then $u\in D\setminus \{0\}$ and $v\in D_+\setminus \{0\}$, $v\ne u$,  implies  $u/v\in V$.
\end{enumerate}

\noindent
Let $\mathcal{E}$ be the class of maps which satisfy $(a)$,$(b)$,$(c)$ and assumption $(d)$:
\begin{enumerate}
\item[(d)]   $f\colon D\to V$ has no singular values in $V\setminus \{0,1\}$
and $c>0$ is minimal such that $f$ has a positive local maximum at $c$ and $f(c)=1$.
\end{enumerate}
\noindent
Similarly let  $\mathcal{E}_o$ be the class of maps which satisfy $(a)$,$(b)$,$(c)$ and assumption $(e)$:
\begin{enumerate}
\item[(e)] $f$ is odd,   $f\colon D\to V$ has no singular values in $V\setminus \{0,\pm 1\}$
  and $c>0$ is minimal such that $f$ has a positive local maximum at $c$ and $f(c)=1$.
\end{enumerate}
Here, as usual, we say that $v\in \C$ is a {\em singular value} of a holomorphic map $f: D\to \C$ if it is a critical value,
or an asymptotic value where the latter means the existence of a path $\gamma: [0,1)\to D$ so that $\gamma(t)\to \partial D$ and $f(\gamma(t))\to v$ as $t\uparrow 1$. Note that we do not require here that $V\supset D$.

Classes $\mathcal{E}$ and $\mathcal{E}_o$ are rich even in the case $D=\C$. See \cite{GO} for a general method of constructing entire (or meromorphic) functions with prescribed asymptotic and critical values.
These classes are also non-empty when $V=\C$ and the domain $D$ is a topological disk or even if $D$ not simply-connected \cite{Er}.
$V$ can also be a bounded subset of $\C$, see example (v) below.


Concrete examples of functions $f$ of the class $\mathcal{E}$ are, where in (i)-(iv) we have  $D=V=\C$,
\begin{enumerate}
\item[(i)] $f(z)=4z(1-z)$,
\item[(ii)]  $f(z)=4\exp(z)(1-\exp(z))$,
\item[(iii)] $f(z)=[\sin(z)]^2$,
\item[(iv)]  $f(z)=m^{-m} (ez)^m\exp(-z)$ when $m$ is a positive even integer.
\item[(v)] the unimodal map $f\colon [0,1]\to [0,1]$ defined by
$$f(x)=\exp(2^\ell) \left( -\exp(-1/|x-1/2|^\ell) + \exp(-2^\ell) \right)$$
satisfies $f(0)=f(1)=0$, $f(1/2)=1$  and  has a
flat critical point at $c=1/2$; this map has an extension $f\colon D\to V$ which is in  $\mathcal E$
and for which $V$ a punctured bounded disc provided $\ell$ is big enough and $D=D_-\cup D_+$
are disjoint open topological discs so that $D_-\cap \R=(0,1/2)$ and $D_+\cap \R=(1/2,1)$, see Lemma~\ref{lemflat}.
\end{enumerate}

Examples  of maps in the class $\mathcal{E}_o$ are
\begin{enumerate}
\item[(vi)] $f(z)=\sin(z)$ and
\item[(vii)] $f(z)= (m/2)^{-m/2}e^{m/2}z^m \exp(-z^2)$ when $m$ is a positive odd integer.
\end{enumerate}


\begin{theorem} Let $f\in \mathcal{E}\cup \mathcal{E}_o$ and for each $w\in W:=D_+$ define $G_w(z)=w\cdot f(z)$ and
${\bf p}(w)=c$. Take  $c_1\in D$,  $g=G_{c_1}$ and assume that there exists $q$ so that
$c_n=g^{n-1}(c_1)\in D$ for all $n\le q$ and either $c_q=c$ or $c_q\in \{c_1,\dots,c_{q-1}\}$.
Moreover, assume $c_n\notin \{c_0,c_1,\dots,c_{n-1}\}$ for $0<n<q$.
Then
\begin{itemize}
\item $(g,G,{\bf p})_W$  satisfies the lifting property and transversality holds.
\item if  $(g,G,{\bf p})_W$ is real, then positive transversality holds.
\end{itemize}
\end{theorem}

\subsection{Applications to monotonicity of  topological entropy of interval maps}
\begin{coro} Take $f\in \mathcal F$ and consider the family $f_c=f+c$, $c\in J= U\cap \R$.
Then the kneading sequence $\mathcal{K}(f_c)$ is monotone increasing in $c\in J$.
Moreover, whenever $c_*\in J$ is so that $f_{c_*}^q(0)=0$ and $f^k_{c_*}(0)\not=0$ for all $1\le k<q$
 the following positive transversality condition
\begin{equation}
Q=\frac{\frac{d}{d c} f_c^q(0)\left.\right|_{c=c_*}}{Df_{c_*}^{q-1}(c_*)}>0  .
\label{eq:trans}
\end{equation}
holds and the topological entropy of $f_c$ is decreasing in $c\in J$.

The same statement holds for $f_c=c\cdot f$ for  $f\in \mathcal E\cup \mathcal{E}_0$, except in this
case we consider the topological entropy of the unimodal map $f|(0,b)$ where  $b=\sup\{b'\in I: b'>0, f(y)>0\,\, \forall y\in (0,b')\}$.
\end{coro}

Monotonicity of  entropy was proved in the case $f_c(x)=x^2+c$  in the 1980s as a major result in unimodal dynamics. By now there are several proofs, see~\cite{MT, Su, D, Tsu0,Tsu1}. All these proofs use complex analytic methods and rely on the fact that $f_c$ extends to a holomorphic map on the complex plane. These methods work well for $f_c(x)=|x|^\ell+c$ when $\ell$ is a positive even integer but break down for general
$\ell>1$ and also for other families of non-analytic unimodal maps.  No approach using purely real-analytic method has so far been successful in proving monotonicity for any $\ell>1$.
The approach to prove monotonicity via the inequality (\ref{eq:trans})  was also previously  used
by Tsujii \cite{Tsu0,Tsu1} for real maps of the form $z\mapsto z^2+c$, $c\in \R$.


%

\begin{remark} Let $\mathcal{U}$ denote the collection of unimodal maps $f: \R\to \R$ which are strictly decreasing in $(-\infty,0]$ and strictly increasing in $[0,\infty)$.
The {\em Milnor-Thurston kneading sequence} of $f\in\mathcal{U}$ is defined as a word  $\mathcal{K}(f)=i_1 i_2\cdots\in \{1, 0, -1\}^{\Z^+}$, where
$$i_k=\left\{\begin{array}{ll}
1 & \mbox{ if } f^k(0)>0\\
0 & \mbox{ if } f^k(0)=0\\
-1 & \mbox{ if } f^k(0)<0.
\end{array}
\right.
$$
For $g\in\mathcal{U}$ with $\mathcal{K}(g)=j_1j_2\cdots$, we say that $\mathcal{K}(f)\prec \mathcal{K}(g)$ if there is some $n\ge 1$ such that
$i_k=j_k$ for all $1\le k<n$ and $\prod_{k=1}^n i_k < \prod_{k=1}^n j_k$.
\end{remark}

\begin{remark}[{\bf Positive transversality and topological entropy}]
Because $f$ has a minimum at $0$,
$x\mapsto f^q_{c_*}(x)$ has a local maximum (minimum) at $0$ if $Df_{c_*}^{q-1}(f_{c_*}(0))<0$ (resp. $>0$). Hence
Equation (\ref{eq:trans}) implies that if $0$ has (precisely) period $q$ at some parameter $c_*$, then
$$\begin{array}{rl} \frac{d}{dc}f^q_c(0)\bigr \vert_{c=c_*}<0 & \mbox { if } f^q_{c_*} \mbox{ has a local maximum at }0,\\
\frac{d}{dc} f^q_c(0)\bigr \vert_{c=c_*}>0 & \mbox { if }  f^q_{c_*} \mbox{ has a local minimum at }0.\end{array}
$$
Hence  the number of laps of $f^n_c$ (and therefore the topological entropy) is non-increasing when $c$ increases.
These inequalities also show that
the multiplier $\lambda(c)$ of the (local) analytic continuation $p(c)$ of this periodic point of period $q$ is strictly increasing.
Note that there is a result of Douady-Hubbard which asserts that in each hyperbolic component of the family of quadratic maps, the
multiplier of the periodic attractor is a univalent function of the parameter.
Proving (\ref{eq:trans}) complements this by
also showing that on the real line the multiplier of the periodic point is increasing.

\end{remark}

When $f_c(x)=|x|^\ell+c$,  and $\ell$ is not an integer, we have not been able able to prove the lifting property.
The next theorem, which will be proved in Appendix A,  gives monotonicity 
when $\ell$ is a large real number (not necessarily an integer),
but only
if not too many points in the critical orbit are in the orientation reversing branch.

\begin{atheorem}
Let  $\ell_-,\ell_+\ge 1$ and consider the family of unimodal maps $f_c=f_{c,\ell_-,\ell_+}$ where
$$f_c(x)=\left\{\begin{array}{ll}
|x|^{\ell_-}+c & \mbox{ if } x\le 0\\
|x|^{\ell_+}+c & \mbox{ if } x\ge 0.
\end{array}
\right.
$$
For any integer $L\ge 1$ there exists $\ell_0>1$
so that for any $q\ge 1$ and any
periodic  sequence   $\bold i=i_1i_2\cdots\in \{-1,0,1\}^{\Z^+}$ of period $q$
so that
\begin{equation}\#\{1\le j< q ; i_j =-1 \}\le L,\label{assum:comb}\end{equation}
and any pair $\ell_-,\ell_+\ge \ell_0$ there is at most one $c\in\R$ for which the kneading sequence
of $f_c$ is equal to $\bold i$. Moreover, if \, $\bold i$\,  is realisable (i.e.  if $c=c_*$ exists) and \, $\bold i$\, has minimal period $q$  then positive transversality holds
\begin{equation}\label{fiortrans}
Q=\dfrac{ \frac{d}{d c}f_c^q(0)\left.\right|_{c=c_*}}{Df_{c_*}^{q-1}(c_*)}>0.
\end{equation}
\end{atheorem}

The proof of this theorem uses delicate geometric arguments,  see Appendix~\ref{sec:finiteorder}.
\ARNOLD{In Section~\ref{sec:finiteoddorder} an analogue of this theorem is proved for the case that $\ell$ is an {\em arbitrary} odd integer,
but under a stronger assumption on the combinatorics of the critical orbit.}
Note that there is an elegant algebraic proof of transversality for critically finite quadratic polynomials in \cite[Chapter 19]{DH1}. This proof also works for $x\mapsto |x|^{2n+1}+c$  provided $n$ is a positive integer, but it does not give the sign, so no monotonicity for this family can be deduced.


%
%


\subsection{Monotonicity along curves with one free critical point}
The above results require that all critical points are eventually periodic.
Nevertheless, they also give information about the bifurcations that occur for example
along a curve  $L_*$  in parameter space corresponding to $(\nu-1)$-critical relations.
The results in Section~\ref{sec:parameter} informally state:

\begin{iitheorem} Critical relations unfold everywhere in the same direction along $L_*$. \end{iitheorem}

This makes it possible to obtain  information about monotonicity
of entropy along the {\em bone curves} considered in \cite[Figure 11]{MTr} and \cite[Figure 8]{Radu}.
Indeed we obtain an alternative proof for one of the main technical steps in \cite{MTr}
in Theorem~\ref{thm:bones}. Could such a simplification be made in the case with at least three critical points?

Indeed, it would be interesting to know whether the sign in (\ref{eq:trans2}) makes it possible to simplify the existing proofs
of Milnor's conjecture. This conjecture is about  the space of real polynomials with only real critical points, all of which non-degenerate,
and asks whether the  level sets of constant topological entropy are connected.
The proof of this conjecture in \cite{MTr} in the cubic case and in \cite{BS} for the general case
relies on quasi-symmetric rigidity, but does having a positive sign in (\ref{eq:trans2}) everywhere allow for a simplification
of the proof of this conjecture?

\subsection{Other applications}

Our approach can also be applied to many other settings, such as
families of Arnol'd maps,  families of  piecewise linear maps
and to families of intervals maps with discontinuities (i.e. Lorenz maps), see
\cite{LSvS2,LSvScompanion}.

Even though we deal with the polynomial and rational case in Appendix C, since it
is so important, in a separate paper \cite{LSvS} we have given a very elementary proof
of transversality and related results in that setting, but without the sign in (\ref{fiortrans}) and  (\ref{eq:trans2}). In that
paper the postcritical set is allowed to be {\em infinite}. See \cite{Ep2} for an alternative discussion on transversality
for maps of finite type,  and \cite{BE} when the  postcritical set is finite.

%
%
%

\medskip

{\bf Acknowledgment.} We are indebted to Alex Eremenko for very helpful discussions concerning Subsection~\ref{subsec:multiplicative}.
The first author acknowledges the support of ISF grant 1226/17 grant, the second author acknowledges the support of NSFC grant no: 11731003,  and the last author
acknowledges the support of ERC AdG grant no: 339523 RGDD. We would also like to thank the referee for some very useful suggestions.


\section{Organisation of this paper and outline of the proof}
 In this paper we consider holomorphic maps $g\colon U\to \C$ where $U$ is an open subset of the complex plane, together with a finite
forward invariant marked set $P$, for example the postcritical set.  These maps do not necessarily have to be rational or transcendental.
The aim is to show that critical relations of such a marked map unfold transversally under a holomorphic deformation $G$ of $g$.
We do this as follows. First,  in Section~\ref{sec:operatorA}, we associate a linear operator $\mathcal A \colon \C^{\# g(P)}\to \mathcal \C^{\# g(P)}$  by the action of  $G$  induced by lifting holomorphic motions on $g(P)$ and show
 \begin{equation*}
 1\notin \mbox{spec}(\mathcal A)   \Leftrightarrow  \mbox{transversality},
 \end{equation*}
 \begin{equation*}
  \,\, \mbox{spec}(\mathcal A) \subset \overline{\mathbb D}\setminus \{1\}   \mbox{ and  } G \mbox{ real }  \implies \mbox{positive transversality}.
 \end{equation*}
More precisely, it is shown in Theorem~\ref{thm:specAtrans}
that the dimension of kernel of $D\mathcal{R}(c_1)$ is equal
 to the  geometric multiplicity of the eigenvalue $1$ of  $\mathcal{A}$.
In Section~\ref{sec:lifting} we then show
 \begin{equation*} \mbox{lifting property} \implies
\mbox{spec}(\mathcal A) \subset \overline{\D} 
 \end{equation*}
Then in Section~\ref{sec:mainthm} we show that provided the lifting property holds,
 $\{w;\mathcal R(w)=0\}$ is locally a   smooth submanifold whose  dimension  is equal to the geometric multiplicity of the eigenvalue $1$ of $\mathcal{A}$.
 In applications, it is usually quite easy to show that the parameter set $\{w;\mathcal R(w)=0\}$  cannot
be a manifold of dimension $>0$, and therefore that $1\not\in \mbox{spec}(\mathcal A)$ and so  transversality holds.

 \bigskip

 It follows that transversality essentially follows from the lifting property.
In Section~\ref{sec:liftingclassical}  we show that the lifting property holds in some classical settings.
%
%
%
%
In Sections~\ref{sec:oneparameter} 
we will show the lifting property holds for polynomial-like mappings from a separation property, and for
maps from the classes $\mathcal{E}$, $\mathcal{E}_o$.
In this way, we derive   transversality for many families of interval maps, for example for a wide class of one-parameter
families of the form $f_\lambda(x)=f(x)+\lambda$ and $f_\lambda(x)=\lambda f(x)$.
As an easy application, we will recover known transversality results for the family of quadratic maps,
and address some conjectures from the 1980's about families of interval maps
of this type.

In Appendix A
we will study the family $x\mapsto |x|^\ell+c$. When $\ell$ is not an even integer, we have not been able to prove the lifting property in general.
Nevertheless we will obtain the lifting property under additional assumptions.

In Appendix B we give some examples for both transversality and the lifting property fails to hold. 
\ARNOLD{ Sections \ref{sec:finiteorder}-\ref{sec:finiteoddorder}.}

In a companion paper
we show that the methods developed in this paper also apply to other families, including
some for which  separation property does not hold,  such as  the   Arnol'd family. 
We also obtain positively oriented transversality for piecewise linear interval maps and interval maps with discontinuities (i.e. Lorenz maps),
see also \cite{LSvS2}

\section{The spectrum of a transfer operator $\mathcal A$ and transversality}
\label{sec:operatorA}


In this section we define a transfer operator $\mathcal A$ associated to the analytic deformation
of a marked map, and show that if  $1$ is not an eigenvalue of $\mathcal A$
then transversality holds. 
If the spectrum of $\mathcal A$ is inside the closed unit circle,  we will obtain additional information
about transversality, see Section~\ref{subsec:trans-sign}.

\subsection{A transfer operator associated to a deformation of a marked map}\label{subsec:21}
In \S\ref{subsec:liftproperty}, we defined lift of holomorphic motions of $g(P)$ associated to $(g,G,\textbf{p})_W$.
Obviously there is a linear map $\mathcal{A}: \C^{\#g(P)}\to \C^{\#g(P)}$ such that whenever $\hat{h}_\lambda$ is a lift of $h_\lambda$, we have
$$\mathcal{A}\left(\left\{\frac{d}{d\lambda}h_\lambda(x)\left|\right._{\lambda=0}\right\}_{x\in g(P)}\right)
=\left\{\frac{d}{d\lambda}\hat{h}_\lambda(x)\left|\right._{\lambda=0}\right\}_{x\in g(P)}.$$
We will call $\mathcal{A}$ the {\em transfer operator} associated to the holomorphic deformation $(g,G, \textbf{p})_W$ of $g$.

If both $g$ and $(g, G, \textbf{p})_W$ are real, then $\mathcal{A}(\R^\nu)\subset \R^\nu$. In this case, we shall often consider real holomorphic motions, i.e. $\Lambda\ni 0$ is symmetric with respect to $\R$ and $h_\lambda(x)\in \R$ for each $x\in g(P)$ and $\lambda\in \Lambda\cap\R$. Clearly, a lift of a real holomorphic motion is again real.

\begin{exam}\label{example1}
Let $g$ be a marked map with $P\supset P_0=\{c_0=0\}$, so that $P=\{c_0,\dots,c_{q-1}\}$,
$c_i=g^i(c_0)$, $0\le i < q$
are distinct and $g^q(c_0)=c_0$.  Consider a deformation $(g,G,\bf p)_W$ where $W$ is a
neighbourhood of $c_1$ and let  ${\bf p}\colon W\to \C$ be so that ${\bf p}\equiv 0$.
Consider the holomorphic motion of $g(P)=P$ defined by $h_\lambda(c_i)=c_i+v_i\lambda$.
Then $\hat h_\lambda(c_i)=c_i+\hat v_i \lambda + O(\lambda^2)$ is defined by $\hat h_\lambda(c_0)=0$,
$G_{h_\lambda(c_1))}(\hat h_\lambda(c_i))=h_\lambda(c_{i+1})=c_{i+1}+v_{i+1}\lambda$,
$i=1,\dots,q-1$ where we take $c_q=c_0$, $v_q=v_0$.
Writing $L_i=\frac{\partial G_w(c_i)}{\partial w}$  and $D_i=Dg(c_i)$
we obtain $$\hat v_0=0, \quad L_i v_1 + D_i\hat v_i = v_{i+1}.$$
So $$\mathcal{A}=\left( \begin{array}{cccccc} 0 & 0 & 0  & 0  & \dots &  0 \\ 0 &  -L_1/D_1  & 1/D_1 & 0 & & 0  \\
0  & -L_2/D_2 & 0 & 1/D_2 & \dots & 0   \\ \vdots & \vdots & \vdots & \vdots & \ddots & \vdots  \\
0 & - L_{q-2}/D_{q-2} & 0 & 0 & \dots & 1/D_{q-2} \\
1/D_{q-1} &  -L_{q-1}/D_{q-1}  & 0 & 0 & \dots &  0
\end{array}\right).$$
Hence
$$I-\rho {\mathcal A}= \left( \begin{array}{cccccc} 1 & 0 & 0  & 0  & \dots &  0 \\ 0 &  1+\rho L_1/D_1  & -\rho/D_1  & 0 & & 0  \\
0  & \rho L_2/D_2 & 1 & -\rho/D_2 & \dots & 0   \\ \vdots & \vdots & \vdots & \vdots & \ddots & \vdots  \\
0 & \rho L_{q-2}/D_{q-2} & 0 & 0 & \dots & -\rho/D_{q-2}  \\
-\rho/D_{q-1}  &  \rho L_{q-1}/D_{q-1}  & 0 & 0 & \dots &  1
\end{array}\right) .$$
So
\begin{equation} \det(I-\rho {\mathcal A})
=1+\frac{L_1}{D_1}\rho + \frac{L_2}{D_1D_2}\rho^2 + \dots +  \frac{L_{q-1}}{D_1D_2\dots D_{q-1}} \rho^{q-1}.
\label{eq:I-rA}
\end{equation}
So if the spectrum of ${\mathcal A}$ is contained in the open unit disc and $L_i, D_i$  are real,
then  (\ref{eq:I-rA}) is strictly positive for all $\rho\in  [-1,1]$.
Note that when $G_w(z)=g(z)+(w-c_1)$,  the expression (\ref{eq:I-rA}) agrees with (\ref{eq:trans}) for $\rho=1$.
\end{exam}


\subsection{Relating the transfer operator with transversality}\label{subsec:transA}

It turns out that transversality is closely related to the eigenvalues of $\mathcal{A}$:

\begin{theorem}\label{thm:specAtrans}
Assume the following holds:
for any $r<j\le \nu$, $Dg^{q_j-l_j}(c_{l_j, j})\not =1$.
Then the following statements are equivalent:
\begin{enumerate}
\item $1$ is an eigenvalue of $\mathcal{A}$;
\item $D\mathcal{R}(\textbf{c}_1)$ is degenerate.
\end{enumerate}
More precisely,  the dimension of kernel of $D\mathcal{R}(c_1)$ is equal
 to the dimension of the
eigenspace of $\mathcal{A}$ associated with eigenvalue $1$.
\end{theorem}
\begin{proof} We first show that (1) implies (2), even without the assumption. So suppose that $1$ is an eigenvalue of $\mathcal{A}$ and let $\textbf{v}=(v(x))_{x\in g(P)}$ be an eigenvector associated with $1$.
For $t\in \D$, define $h_t(x)=x+t v(x)$ for each $x\in g(P)$ and $w(t)=(c_{1,j}+ t v(c_{1,j}))_{j=1}^\nu$.
Then for each $x\in g(P)\setminus P_0$,
\begin{equation}\label{eqn:21.1}
G_{w(t)}(h_t(x))-h_t(g(x))=O(t^2),
\end{equation} and for each $x=c_{0,j}\in g(P)\cap P_0$, we have
\begin{equation}\label{eqn:21.2}
h_t(x)-p_j(w(t))=O(t^2).
\end{equation}
For each $1\le j\le \nu$, and each $1\le k<q_j$,
applying (\ref{eqn:21.1}) repeatedly, we obtain
\begin{equation}\label{eqn:21.3}
G_{w(t)}^k(h_t(c_{1,j}))=h_t(g^k(c_{1,j}))+O(t^2).
\end{equation}
Together with (\ref{eqn:21.2}), this implies that $$R_j(w(t))=O(t^2),$$
holds for all $1\le j\le \nu$. It remains to show $w'(0)\not=\textbf{0}$. Indeed, otherwise, by (\ref{eqn:21.3}), it would follow that $v(g^k(c_{1,j}))=(g^k)'(c_{1,j}) v(c_{1,j})=0$ for each $1\le j\le \nu$ and $1\le k<q_j$, and hence $v(x)=0$ for all $x\in g(P)$, which is absurd. We completed the proof that (1) implies (2).

\medskip

Now let us prove that (2) implies (1) under the assumption of the lemma.
Suppose that $D\mathcal{R}(\textbf{c}_1)$ is degenerate. Then there exists a non-zero vector $(w_1^0, w_2^0,\cdots, w_\nu^0)$ in $\C^\nu$ such that $R_j(w(t))=O(t^2)$ as $t\to 0$ for all $j=1,\dots,\nu$,  where $w(t)=(w_j(t))_{j=1}^\nu=(c_{1,j}+ tw_j^0)_{j=1}^\nu$.  We
claim that $w_j^0=w_{j'}^0$ holds whenever $c_{1,j}=c_{1,j'}$, $1\le j,j'\le \nu$. Indeed,

{\em Case 1.} If $1\le j\le r$ then $1\le j'\le r$ and $\mu(j)=\mu(j')$, $q_j=q_{j'}$. Then
$$G_{w(t)}^{q_j-1}(w_j(t))-G_{w(t)}^{q_j-1}(w_{j'}(t))=R_j(w(t))-R_{j'}(w(t))=O(t^2)$$
which implies that $w_j(t)-w_{j'}(t)=O(t^2)$, i.e. $w_j^0=w_{j'}^0$.

{\em Case 2.} If $r<j\le \nu$ then $r<j'\le \nu$ and   $\hat l_j=\hat l_{j'}$, $\hat q_j=\hat q_{j'}$
where we define for any $r<j\le \nu$ the integers $\hat{l}_j<\hat{q}_j$ minimal  so that  $g^{\hat q_j}(c_{0,j})=g^{\hat l_j}(c_{0,j})$. By the chain rule it follows that
$$G_{w(t)}^{q_j-1}(w_j(t))-G_{w(t)}^{l_j-1}(w_j(t))=O(t^2)$$
implies
$$G_{w(t)}^{\hat{q}_j-1}(w_j(t))-G_{w(t)}^{\hat{l}_j-1}(w_j(t))=O(t^2)$$
under the assumption $Dg^{q_j-l_j}(c_{l_j,j})\not=1$.

Thus we obtain
$$G_{w(t)}^{\hat q_j-1}(w_j(t))-G_{w(t)}^{\hat q_{j}-1}(w_{j'}(t))=
G_{w(t)}^{\hat l_j-1}(w_j(t))-G_{w(t)}^{\hat l_{j}-1}(w_{j'}(t))+O(t^2),$$
which implies that
$$(Dg^{\hat q_j-1}(c_{l_j,j})-Dg^{\hat l_j-1}(c_{l_j,j}))(w_j(t)-w_{j'}(t))=O(t^2).$$
If such $j$ and $j'$  exist then $c_{l_j,j}$ is a hyperbolic periodic point,
hence $Dg^{\hat q_j-1}(c_{l_j,j})\not=Dg^{\hat l_j-1}(c_{l_j,j})$.
It follows that $w_{j}^0=w_{j'}^0$.

Thus the Claim is proved.
To obtain an eigenvector for $\mathcal{A}$ with eigenvalue $1$, define $v(c_{1,j})=w_j^0$,
$v(c_{0,j})=\frac{d}{dt} p_j(w(t))|_{t=0}$. For points $x\in g(P)\setminus P_0$, there is $j$ and $1\le s<q_j$ such that $x=g^s(c_{0,j})$, define $v(x)=\frac{d}{dt} G_{w(t)}^{s-1}(w_j(t))|_{t=0}$. Note that $v(x)$ does not depend on the choice of $j$ and $s$. (This can be proved similarly as the claim.)

The above argument builds an isomorphism between   $\{v\in\C^\nu: D\mathcal{R}(\textbf{c}_1, v)=0\}$
and the eigenspace of $\mathcal{A}$ associated with eigenvalue $1$. So these two spaces have the same dimension.
\end{proof}

\subsection{The spectrum of $\mathcal A$ and the determinant of some matrix $D(\rho)$}
\label{subsec:trans-sign}
Define $D(\rho)=(D_{j,k}(\rho))_{1\le j,k\le \nu}$ as follows: Put
$$L_{k}(z)=\frac{\partial G_{w}(z)}{\partial w_k}\left.\right|_{w=\textbf{c}_1};\quad
p_{j,k}=\frac{\partial p_j}{\partial w_k}(\textbf{c}_1);$$
$$\mathcal{L}^0_{j,k}=0 \text{ and } \mathcal{L}^m_{j,k}=\sum_{n=1}^m\frac{\rho^n L_k(c_{n,j})}{Dg^n(c_{1,j})} \text{ for } m>0;$$
$$D_{jk}(\rho)=\delta_{jk} +\mathcal{L}^{q_j-1}_{j,k}-\rho^{q_j} \frac{p_{\mu(j),k}}{Dg^{q_j-1}(c_{1,j})}$$
when $1\le j\le r$ and
$$D_{jk}(\rho)=\delta_{jk}+\mathcal{L}^{q_j-1}_{j,k}
-\frac{\rho^{q_j-l_j}}{Dg^{q_j-l_j}(c_{l_j,j})}\left(\mathcal{L}^{l_j-1}_{jk}+\delta_{j,k}\right)$$
when $r<j\le \nu$. Note that
\begin{equation} \det(D\mathcal{R}(\textbf{c}_1))=\prod_{j=1}^\nu Dg^{q_j-1}(c_{1,j})\det(D(1)).\label{eq:Drho} \end{equation}


We say that $\rho\in \C$ is an {\em exceptional value} if
there exists $r<j\le \nu$ such that $Dg^{q_j-l_j}(c_{l_j,j})=\rho^{q_j-l_j}$.

\begin{prop}\label{prop:non-exceptional}
For each non-exceptional $\rho\in \C$, we have
\begin{equation}
\det(I-\rho\mathcal{A})=0\Leftrightarrow \det (D(\rho))=0.\label{eq:non-exceptional}\end{equation}

\end{prop}
\begin{proof} For $\rho=0$, $\det(I)=\det(D(0))=1$. Assume $\rho\not=0$.
Define a new triple $(g^\rho, G^\rho, \textbf{p}^\rho)$ as follows.
\begin{itemize}
\item For each $x\in P\setminus P_0$, $G^\rho_{w}(z)=G_{w}(x)+\frac{Dg(x)}{\rho}(z-x)$ in a neighbourhood of $x$;
\item $g^\rho(x)=g(x)$ for each $x\in P_0$ and $g^\rho(z)= G_{\textbf{c}_1}^\rho(z)$ in a neighbourhood of $P\setminus P_0$;
\item $\textbf{p}^\rho(w)= \textbf{c}_0+ \rho \frac{\partial \textbf{p}}{\partial w}(\textbf{c}_1) \cdot (w-\textbf{c}_1).$
\end{itemize}
Let $\mathcal{A}^\rho$ be the transfer operator associated with the triple $(g^\rho, G^\rho, \textbf{p}^\rho)$. Then it is straightforward to check that
$$\mathcal{A}^\rho=\rho \mathcal{A}.$$
We can define a map $\mathcal{R}^\rho=(R^\rho_1, R^\rho_2,\cdots, R_\nu^\rho)$ for each $\rho\not=0$ in the obvious way:
$$R^\rho_j(w)=(G_{w}^\rho)^{q_j-1}(w_j)-p^\rho_{\mu(j)}(w)$$
for $1\le j\le r$ and
$$R^\rho_j(w)=(G_{w}^\rho)^{q_j-1}(w_j)-(G_{w}^\rho)^{l_j-1}(w_j)$$
for $r<j\le \nu$.
As long as $\rho$ is non-exceptional for the triple $(g,G, \textbf{p})$, the new triple $(g^\rho, G^\rho, \textbf{p}^\rho)$ satisfies the assumption of Theorem~\ref{thm:specAtrans}, thus
$$\det(I-\mathcal{A}^\rho)=0\Leftrightarrow D\mathcal{R}^\rho(\textbf{c}_1) \text{ is degenerate}.$$

Direct computation shows that the $(j,k)$-th entry of
$D\mathcal{R}^\rho (\textbf{c}_1)$ is equal to
$D_{j,k}(\rho)Dg^{q_j-1}(c_{1,j})/\rho^{q_j-1}.$ Indeed,
for each $1\le j\le r$,
\begin{align*}
D^\rho_{j,k}(\rho)&=\frac{\partial (G^\rho_{w})^{q_j-1}(c_{1,j})}{\partial w_k}\left.\right|_{w=\textbf{c}_1}+\frac{Dg^{q_j-1}(c_{1,j})}{\rho^{q_j-1}}\delta_{jk}- \rho \frac{\partial p_{\mu(j)}}{\partial w_k}\\
&=\frac{Dg^{q_j-1}(c_{1,j})}{\rho^{q_j-1}}\left(\delta_{jk}+\sum_{n=1}^{q_j-1} \frac{\rho^n L_k(c_{n,j})}{Dg^n(c_{1,j})}-\rho^{q_j}\frac{p_{\mu(j),k}}{Dg^{q_j-1}(c_{1,j})}\right),
\end{align*}
and for $r<j\le \nu$,
\begin{align*}
& D^\rho_{jk}(\rho)\\
=&
\frac{\partial ((G_{w}^\rho)^{q_j-1}(c_{1,j})-(G_{w}^\rho)^{l_j-1}(c_{1,j}))}{\partial w_k}\left.\right|_{w=\textbf{c}_1}+\delta_{jk}
\left(\frac{Dg^{q_j-1}(c_{1,j})}{\rho^{q_j-1}}-\frac{Dg^{l_j-1}(c_{1,j})}{\rho^{l_j-1}}\right)\\
=&\frac{Dg^{q_j-1}(c_{1,j})}{\rho^{q_j-1}}\mathcal{L}^{q_j-1}_{j,k}
-\frac{Dg^{l_j-1}(c_{1,j})}{\rho^{l_j-1}}\mathcal{L}^{l_j-1}_{j,k}+\delta_{jk}
\left(\frac{Dg^{q_j-1}(c_{1,j})}{\rho^{q_j-1}}-\frac{Dg^{l_j-1}(c_{1,j})}{\rho^{l_j-1}}\right)
\end{align*}
Therefore $\det(I-\rho \mathcal{A})=0$ if and only if $\det(D(\rho))=0$.
\end{proof}

\subsection{Positive transversality in the real case}

To illustrate the power of the previous proposition we state:

\begin{coro}[Positive transversality]
\label{real} Let $(g, G, \textbf{p})_W$ be a real local holomorphic deformation of a real marked map $g$. Assume that one has $|Dg^{q_j-l_j}(c_{l_j,j})|>1$ for all $r<j\le \nu$.
Assume furthermore that all the eigenvalues of $\mathcal{A}$ lie in the set $\{|\rho|\le 1, \rho\not=1\}$.
Then the `positively oriented' transversality condition holds.

\end{coro}
\begin{proof} Write the polynomial $\det (D(\rho))$ in the form $\prod_{i=1}^N (1-\rho \rho_i)$, where $\rho_i\in \C\setminus \{0\}$.
Because of (\ref{eq:Drho}) it suffices to show that
$\det( D(1))>0$. Since $\det(D(\rho))$ is a real polynomial in $\rho$, this follows from
$\rho_i\not\in [1,\infty)$ for each $i$. Arguing by contradiction, assume that $\rho_i\ge 1$ for some $i$. Then $1/\rho_i$ is a zero of $\det(D(\rho))$. As $|1/\rho_i|
\le 1$, $1/\rho_i$ is not an exceptional value. Thus
$\det(I-\mathcal{A}/\rho_i)=0$, which implies that $\rho_i\ge 1$ is an eigenvalue of $\mathcal{A}$, a contradiction!
\end{proof}

\subsection{A remark on an alternative transfer operator}
Proposition~\ref{prop:non-exceptional} shows that for non-exceptional $\rho$,
one has (\ref{eq:non-exceptional}). One can also associate to $(g,G, \textbf{p})$ another linear operator $\mathcal{A}_J$ for which
\begin{equation}\label{I}
\det D(\rho)=\det (I-\rho \mathcal{A}_J)
\end{equation}
holds for {\em all} $\rho\in \C$.
Here $J$ denotes a collection of all pairs $(i,j)$ such that $1\le j\le \nu$,  $0\le i\le q_j-1$ and if $i=0$ then $j=\mu(j')$ for some $1\le j'\le \nu$. Given a collection of functions $\{c_{i,j}(\lambda)\}_{(i,j)\in J}$ which are holomorphic in a small neighbourhood of $\lambda=0$, there is another collection of holomorphic near $0$ functions $\{\hat c_{i,j}(\lambda)\}_{(i,j)\in J}$ such that
$\hat c_{0,j}(\lambda)=p_j(\textbf{c}_1(\lambda))$ where $\textbf{c}_1(\lambda)=(c_{1,1}(\lambda),\cdots, c_{1,\nu}(\lambda))$
and, for $i\not=0$, $G(\text{c}_1(\lambda), \hat c_{i,j})=c_{i+1,j}(\lambda)$. Here we set
$c_{q_j,j}(\lambda)=c_{0,\mu(j)}(\lambda)$ for $1\le j\le r$ and $c_{q_j,j}(\lambda)=c_{l_j,j}(\lambda)$ for $r<j\le \nu$.
Define the linear map
$\mathcal{A}_J: \C^{\#J}\to \C^{\#J}$ by taking the derivative at $\lambda=0$: $\mathcal{A}_J(\{c_{i,j}'(0)\}_{(i,j)\in J})=\{\hat c_{i,j}'(0)\}_{(i,j)\in J}$.
Explicitly, we get:
$$\hat c'_{i,j}(0)=\left\{
\begin{array}{ll}
\sum_{k=1}^\nu p_{j,k} &\mbox{ if } i=0 \mbox{ and } j=\mu(j') \mbox{ for some } j'\\
\frac{1}{Dg(c_{i,j})} \left(v_{i+1,j}-\sum_{k=1}^\nu L_k(c_{i,j})v_{1,k}\right) & \mbox{ if } 1\le i<q_j-1, 1\le j\le \nu\\
\frac{1}{Dg(c_{q_j-1,j})} \left(v_{0,\mu(j)}-\sum_{k=1}^\nu L_k(c_{q_j-1,j})v_{1,k}\right) &\mbox{ if } i=q_j-1, 1\le j\le r\\
\frac{1}{Dg(c_{q_j-1,j})} \left(v_{l_j,j}-\sum_{k=1}^\nu L_k(c_{q_j-1,j})v_{1,k}\right) & \mbox{ if } i=q_j-1, r< j\le \nu
\end{array}
\right.
$$
Elementary properties of determinants being applied to the matrix $I-\rho \mathcal{A}_J$ lead to~(\ref{I}).
Observe that $\mathcal{A}_J=\mathcal{A}$ if (and only if) all points $c_{i,j}$, $(i,j)\in J$ are pairwise different.
Therefore, we have:
$$\det (I-\rho \mathcal{A})=\det D(\rho)$$
for every $\rho\in \C$
provided $\sum_{j=1}^\nu (q_j-1)+r=\#P$.

\section{The lifting property and the spectrum of $\mathcal A$}\label{sec:lifting}

The next proposition shows that the lifting property implies that the spectrum of $\mathcal{A}$ is in the closed unit disc.
\begin{prop}
 \label{prop:lift2spectrum}
If $(g,G,\textbf{p})_W$ has the lifting property, then
the spectral radius of the associated transfer operator $\mathcal{A}$ is at most $1$ and every eigenvalue of $\mathcal{A}$ of modulus one is semisimple (i.e. its algebraic multiplicity coincides with its geometric multiplicity). Moreover, for $(g, G, \textbf{p})_W$ real, we only need to assume that the lifting property with respect to real holomorphic motions.
\end{prop}
\begin{proof}
For any $\textbf{v}=(v(x))_{x\in g(P)}$, construct a holomorphic motion $h_\lambda^{(0)}$ over $(\Lambda,0)$ for some domain $\Lambda\ni 0$, such that
$\frac{d}{d\lambda}h^{(0)}_\lambda(x)\left|\right._{\lambda=0}=v(x)$ for all $x\in g(P)$. Then
$$\mathcal{A}^k(\textbf{v})=\left(\frac{d}{d\lambda}h^{(k)}_\lambda(x)\left|\right._{\lambda=0}\right)_{x\in g(P)}$$ for every $k>0$.
By Cauchy's integral formula, there exists $C=C(M, \eps)$ such that $|\frac{d}{d\lambda}h^{(k)}_\lambda(x)\left|\right._{\lambda=0}|\le C$ holds for all $x\in g(P)$ and all $k$. It follows that for any $\textbf{v}\in \C^{\#g(P)}$, $\mathcal{A}^k(\textbf{v})$ is a bounded sequence. Thus the
spectral radius of $\mathcal{A}$ is at most one and every eigenvalue of $\mathcal{A}$ of modulus one is semisimple.

Suppose $(g, G, \textbf{p})_W$ is real. Then for any $\textbf{v}\in \R^{\#g(P)}$, the holomorphic motion $h_\lambda^{(0)}$ can be chosen to be real. Thus if $(g, G,\textbf{p})_W$ has the real  lifting property, then $\{\mathcal{A}^k(\textbf{v})\}_{k=0}^\infty$ is bounded for each $\textbf{v}\in\R^{\#g(P)}$. The conclusion follows.
\end{proof}

To obtain that the radius is strictly smaller than one, we shall apply the argument to a suitable perturbation of the map $g$.
For example, we have the following:

\begin{prop}[Robust spectral property] \label{prop:perturb2spectrum}
Let $(g,G,\textbf{p})_W$ be as above.
Let $Q$ be a polynomial such that
$Q(c_{0,j})=0$ for $1\le j\le \nu$ and $Q(x)=0$, $Q'(x)=1$ for every $x\in g(P)$. Let $\varphi_\xi(z)=z-\xi Q(z)$ and for $\xi\in (0,1)$ let
$\psi_{\xi}(w)=(\varphi_\xi^{-1}(w_1),\cdots,\varphi_\xi^{-1}(w_\nu))$
be a map from a neighbourhood of $\textbf{c}_1$ into a neighbourhood of $\textbf{c}_1$. Suppose that there exists $\xi\in (0,1)$ such that the triple
$(\varphi_\xi\circ g, \varphi_\xi\circ G, \textbf{p}\circ \psi_\xi)$ has the lifting property. Then
the spectral radius of $\mathcal{A}$ is at most $1-\xi$.
\end{prop}
\begin{proof}
Note that $\tilde{g}:=\varphi_\xi\circ g$ is a marked map with the same sets $P_0\subset P$. Furthermore,
$\tilde{g}^i(c_{0,j})=g^i(c_{0,j})=c_{i,j}$, $D\widetilde{g}(c_{i,j})=(1-\xi) Dg(c_{i,j})$,
$\frac{\partial \varphi_\xi\circ G}{\partial w_k}({\bf c}_1,z)=(1-\xi)\frac{\partial G}{\partial w_k}({\bf c}_1,z)$ for each $z\in g(P)\setminus P_0$,
and $\frac{\textbf{p}\circ \psi_\xi}{\partial w_k}({\bf c}_1)=(1-\xi)^{-1}\frac{\partial \textbf{p}}{\partial w_k}({\bf c}_1)$.
Therefore, the operator which is associated to the triple $(\varphi_\xi\circ g, \varphi_\xi\circ G, \textbf{p}\circ \psi_\xi)$ is equal to
$(1-\xi)^{-1}\mathcal{A}$,
Since the latter triple has the lifting property, by Proposition~\ref{prop:lift2spectrum},
the spectral radius of $(1-\xi)^{-1}\mathcal{A}$ is at most $1$.
\end{proof}

For completeness we include:

\begin{lemma} Assume that the spectrum radius of $\mathcal{A}$ is strictly less than $1$. Then the lifting property holds.
\end{lemma}
\begin{proof} 
Let $\Phi(Z)=(\varphi_x(Z))_{x\in g(P)}$ be the holomorphic map defined from a neighbourhood $V$ of the point
${\bf z}:=g(P)\in \C^{\#g(P)}$ by
$$G_{Z_1}(\varphi_{x}(Z))=z_{g(x)}, \ \ \ x\in g(P)\setminus P_0, $$
$$\varphi_{c_{0,j}}(Z)=p_j(Z_1), \ \ \ 1\le j\le \nu ,$$
where ${\bf z}_1=(z_{c_{1,j}})_{j=1}^\nu$, $Z=(z_x)_{x\in g(P)}$.
Then $\Phi({\bf z})={\bf z}$.  Moreover, if $h_\lambda$ is a holomorphic motion of $g(P)$ over $(\Lambda,0)$ with $\textbf{h}_\lambda=(h_\lambda(x))_{x\in g(P)}\in V$ for all $\lambda$, then $\widehat{h}_\lambda(x):=(\Phi(\textbf{h}_\lambda))_x$ is the lift of $h_\lambda$ over $(\Lambda,0)$.

So the derivative of $\Phi$ at ${\bf z}$ is equal to $\mathcal{A}$, and hence ${\bf z}$ is a hyperbolic attracting fixed point of $\Phi$. Therefore, there exist $N>0$ and a neighborhood $\mathcal{U}$ of $\textbf{z}$ such that $\Phi^N$ is well-defined on $\mathcal{U}$
and such that $\Phi^N(\mathcal{U})$ is compactly contained in $\mathcal{U}_0$. It follow $\Phi^n$ converges uniformly to the constant ${\bf z}$ in $\mathcal{U}$.

Let us prove that $(g,G,\textbf{p})_W$ has the lifting property. Indeed, if $h_\lambda$ is a holomorphic motion of $g(P)$ over $(\D,0)$, then there exists $\eps>0$ such that $\textbf{h}_\lambda:=(h_\lambda(x))_{x\in g(P)}\in \mathcal{U}$, so that $\textbf{h}^{(k)}_\lambda:=\Phi^k(\textbf{h}_\lambda)$ is well-defined. Let $h_\lambda^{(k)}(x)$, $x\in g(P)$, be such that
$\textbf{h}^{(k)}_\lambda=(h_\lambda^{(k)}(x))_{x\in g(P)}$. Then for each $k\ge 0$, $h_\lambda^{(k)}$ is a holomorphic motion of $g(P)$ over $(\D_\eps,0)$ and $h_\lambda^{(k+1)}$ is the lift of $h_\lambda^{(k)}$.
%
\end{proof}

\section{The lifting property and persistence of critical relations}
\label{sec:mainthm}
The main technical result in this paper is the following theorem:
\begin{theorem}\label{thm:1eigen}
Assume that either the triple $(g, G, \textbf{p})_W$ has the lifting property or $(g, G,\textbf{p})_W$ is real and has the real lifting property. Assume also that for all $r<j\le \nu$, $Dg^{q_j-l_j}(c_{l_j, j})\not=1$.
Then
\begin{enumerate}
\item All eigenvalues of $\mathcal{A}$ are contained in $\overline{\D}$.
\item There is a neighborhood $W'$ of $\textbf{c}_1$ in $W$ such that
\begin{equation}\label{allornoth}
\{\textbf{w}\in W'  \, | \,  \mathcal{R}(\textbf{w})=0\}
\end{equation}
is a smooth submanifold of $W'$, and its dimension  is equal to the geometric multiplicity of the eigenvalue $1$ of $\mathcal{A}$.
\end{enumerate}
\end{theorem}
The second statement is useful to conclude that $D\mathcal{R}$ is non-degenerate at $\textbf{c}_1$, or equivalently, that $1$ is not an eigenvalue of $\mathcal{A}$. Indeed,
if $\nu=1$ and if $1$ is an eigenvalue of $\mathcal{A}$, the manifold (\ref{allornoth}) must contain a neighbourhood of ${\bf c}_1$ and hence $\mathcal{R}(\textbf{w})=0$ holds for every ${\bf w}\in \C$ near ${\bf c}_1\in \C$, which only happens for trivial family $(g,G,\textbf{p})_W$.
\ARNOLD{A more subtle application of this statement is used in Section \ref{sec:sinearnold}.}
It is also possible to apply this statement in a more subtle way, see \cite{LSvScompanion}.

Let $\Lambda$ be a domain in $\C$ which contains $0$.
A holomorphic motion $h_\lambda(x)$ of $g(P)$ over $(\Lambda,0)$ is called {\em asymptotically invariant of order $m$} (with respect to $(g,G,\textbf{p})_W$) if there is a subdomain $\Lambda_0\subset \Lambda$ which contains $0$ and a holomorphic motion  $\widehat{h}_\lambda(x)$ which is the lift of $h_\lambda$ over $(\Lambda_0,0)$, such that
\begin{equation}
\widehat{h}_\lambda(x)-h_\lambda(x)=O(\lambda^{m+1})\text{ as } \lambda\to 0.
\label{eq:asymptm}\end{equation}
Obviously,
\begin{lemma} $1$ is an eigenvalue of $\mathcal{A}$ if and only if there is a non-degenerate holomorphic motion which is invariant of order $1$.
\end{lemma}
Here, a holomorphic motion $h_\lambda(x)$ is called {\em non-degenerate} if $\frac{d}{d\lambda}h_\lambda(x)\left|\right._{\lambda=0}\not=0$ holds for some $x\in g(P)$.

A crucial step in proving this theorem is the following Lemma~\ref{lem:arbitraryasyminv} whose proof requires the following easy fact
and its corollary. 
\begin{fact}\label{ref:fact} Let $F:\mathcal{U}\to \C$ be a holomorphic function defined in an open set $\mathcal{U}$ of $\C^N$, $N\ge 1$. Let $\gamma, \widetilde{\gamma}:\D_\eps\to \mathcal{U}$ be two holomorphic curves such that $$\gamma(\lambda)-\widetilde{\gamma}(\lambda)=O(\lambda^{m+1})\text{ as }\lambda\to 0.$$
Then
$$F(\gamma(\lambda))-F(\widetilde{\gamma}(\lambda))=\sum_{i=1}^N \frac{\partial F}{\partial z_i}(\gamma(0))(\gamma_i(\lambda)-\widetilde{\gamma}_i(\lambda))+O(\lambda^{m+2})\text{ as }\lambda\to 0.$$
\end{fact}
\begin{proof} For fixed $\lambda$ small, define $\delta(t)=(1-t)\widetilde{\gamma}(\lambda)+t\gamma(\lambda)$ and $f(t)=F(\delta(t))$. Then
$$f'(t)=\sum_{i=1}^N \frac{\partial F}{\partial z_i}(\delta(t)) (\gamma_i(\lambda)-\widetilde{\gamma_i}(\lambda)).$$
Since $\delta(t)-\gamma(0)=O(\lambda)$, and
$F(\gamma(\lambda))-F(\widetilde{\gamma}(\lambda))=\int_0^1 f'(t)dt$, the equality follows.
\end{proof}

\begin{coro}\label{asym}
A holomorphic motion $h_\lambda$ of $g(P)$ is asymptotically invariant of order $m$ if and only if
\begin{enumerate}
\item [(1)] For each $x\in g(P)\setminus P_0$,
$$G_{h_\lambda(c_{1,1}),\cdots, h_\lambda(c_{1,\nu})}(h_\lambda(x))=h_\lambda(g(x)) +O(\lambda^{m+1}) \text{ as } \lambda\to 0.$$
\item [(2)] For $x=c_{0,j}\in g(P)$,
$$p_{j}(h_\lambda(c_{1,1}),\cdots, h_\lambda(c_{1,\nu}))=h_\lambda (c_{0,j})+ O(\lambda^{m+1}) \text{ as }\lambda\to 0.$$
\end{enumerate}
\end{coro}
\begin{proof} If $h_\lambda$ is asymptotically invariant of order $m$, we get (1) applying  Fact~\ref{ref:fact} to the
function $F(z_1,z_2,\cdots, z_\nu, z_{\nu+1})=G_{(z_1,z_2,\cdots, z_\nu)}(z_{\nu+1})$, and we get (2) applying it to
the function $F(z_1,z_2,\cdots, z_\nu)=p_j(z_1,z_2,\cdots, z_\nu)$. Vice versa, assume that (1)-(2) hold.
Given $x\in g(P)\setminus P_0$, let $F(z_1,z_2,\cdots, z_\nu, z_{\nu+1})$ be a local branch of $G_{(z_1,z_2,\cdots, z_\nu)}^{-1}(z_{\nu+1})$ which is a well defined holomorphic function in a neighborhood of $(\textbf{c}_1,g(x))$.
Let $V(\lambda)=G_{h_\lambda(c_{1,1}),\cdots, h_\lambda(c_{1,\nu})}(h_\lambda(x))$ and
$\hat V(\lambda)=h_\lambda(g(x))$. By (1), $V(\lambda)-\hat V(\lambda)=O(\lambda^{m+1})$.
Hence, by Fact \ref{ref:fact},
$$h_\lambda(x)-\hat h_\lambda(x)=F(h_\lambda(c_{1,1}),\cdots, h_\lambda(c_{1,\nu}), V(\lambda))-
F(h_\lambda(c_{1,1}),\cdots, h_\lambda(c_{1,\nu}), \hat V(\lambda))=O(\lambda^{m+1}).$$
For $x=c_{0,j}$, the claim is straightforward.
\end{proof}

\begin{lemma}\label{lem:arbitraryasyminv}
One has the following:
\begin{enumerate}
\item Assume $(g, G,\textbf{p})_W$ has the lifting property.
Suppose that there is a holomorphic motion $h_\lambda$ of
$g(P)$ over $(\Lambda, 0)$ which is asymptotically invariant of order $m$ for some $m\ge 1$. Then there is a non-degenerate holomorphic motion $H_\lambda$ of $g(P)$ over some $(\tilde\Lambda, 0)$ which is asymptotically invariant of order $m+1$. Besides,
$H_\lambda(x)-h_\lambda(x)=O(\lambda^{m+1})$ as $\lambda\to 0$ for all $x\in g(P)$.
\item Assume $(g, G,\textbf{p})_W$ is real and has the real lifting property. Suppose that
there is a real holomorphic motion $h_\lambda$ of
$g(P)$ over $(\Lambda, 0)$ which is asymptotically invariant of order $m$ for some $m\ge 1$. Then there is a non-degenerate real holomorphic motion $H_\lambda$ of $g(P)$ over some $(\tilde\Lambda, 0)$ which is asymptotically invariant of order $m+1$. Besides,
$H_\lambda(x)-h_\lambda(x)=O(\lambda^{m+1})$ as $\lambda\to 0$ for all $x\in g(P)$.
\end{enumerate}
\end{lemma}
\begin{proof} We shall only prove the first statement as the proof of the second is the same with obvious change of terminology. Let $h_\lambda$ be a non-degenerate holomorphic motion of
$g(P)$ over $(\Lambda, 0)$  which is asymptotically invariant of order $m$.
%
By assumption that $(g, G, \textbf{p})_W$ has the lifting property , there exists a smaller domain $\Lambda_0\subset \Lambda$ and holomorphic motions $h^{(k)}_\lambda$ over $\Lambda_0$, $k=0,1,\ldots$ such that
$h^{(0)}_\lambda=h_\lambda$ and such that $h^{(k+1)}_\lambda$ is the lift of
$h^{(k)}_\lambda$ over $(\Lambda_0,0)$ for each $k\ge 0$. Moreover, the functions $h^{(k)}_\lambda$ are uniformly bounded. For each $k\ge 1$, define
$$\psi_\lambda^{(k)}(x)=\frac{1}{k}\sum_{i=0}^{k-1}h_\lambda^{(i)}(x),$$
and $$\varphi_\lambda^{(k)}(x)=\frac{1}{k}\sum_{i=1}^k h_\lambda^{(i)}(x).$$
By shrinking $\Lambda_0$, we may assume that there exists $k_n\to\infty$, such that $\psi_\lambda^{(k_n)}(x)$ converges uniformly in $\lambda\in\Lambda_0$ as $k_n\to\infty$ to a holomorphic function $H_\lambda(x)$. Shrinking $\Lambda_0$ furthermore if necessary, $H_\lambda$ defines a holomorphic motion of $g(P)$ over $(\Lambda_0,0)$. Clearly, $\varphi_\lambda^{(k_n)}(x)$ converges uniformly to $H_\lambda(x)$ as well.

Let us show that $H_\lambda$ is asymptotically invariant of order $m+1$ by applying the fact above. 
Due to Corollary~\ref{asym} (and taking $k=k_n\to \infty$ in the next equations)  this amounts to showing:
\begin{enumerate}
\item [(i)] For each $x\in g(P)\setminus P_0$, and any $k\ge 1$,
$$G_{\psi^{(k)}_\lambda(c_{1,1}),\cdots, \psi_\lambda^{(k)}(c_{1,\nu})}(\varphi_\lambda^{(k)}(x))=\psi_\lambda^{(k)}(g(x)) +O(\lambda^{m+2}) \text{ as } \lambda\to 0.$$
\item [(ii)] For $x=c_{0,j}\in g(P)$,
$$p_{j}(\psi^{(k)}_\lambda(c_{1,1}),\cdots, \psi_\lambda^{(k)}(c_{1,\nu}))=\varphi^{(k)}_\lambda (c_{0,j})+ O(\lambda^{m+2}) \text{ as }\lambda\to 0.$$
\end{enumerate}

Let us prove (i). Fix $x\in g(P)\setminus P_0$ and $k\ge 1$. Let $F(z_1,z_2,\cdots, z_\nu, z_{\nu+1})=G_{(z_1,z_2,\cdots, z_\nu)}(z_{\nu+1})$. By the construction of $h^{(k)}_\lambda$,
we have
$$F(h_\lambda^{(i)}(c_{1,1}),\cdots, h_\lambda^{(i)}(c_{1,\nu}), h_\lambda^{(i+1)}(x))=h_\lambda^{(i)}(g(x))$$
for every $i\ge 0$. Thus
\begin{equation}\label{eqn:arbitraryasyminv1}
\psi^{(k)}_\lambda(g(x))=\frac{1}{k} \sum_{i=0}^{k-1}F(h_\lambda^{(i)}(c_{1,1}),\cdots, h_\lambda^{(i)}(c_{1,\nu}), h_\lambda^{(i+1)}(x)).
\end{equation}
Since all the functions $h_\lambda^{(i)}(x)$,$\psi^{(k)}_\lambda(x)$, $\varphi^{(k)}_\lambda(x)$ have the same derivatives up to order $m$ at $\lambda=0$, applying Fact 6.2, we obtain
\begin{multline*}
F(h_\lambda^{(i)}(c_{1,1}),\cdots, h_\lambda^{(i)}(c_{1,\nu}), h_\lambda^{(i+1)}(x))-F(\psi_\lambda^{(k)}(c_{1,1}),\cdots, \psi^{(k)}_\lambda(c_{1,\nu}),\varphi^{(k)}_\lambda(x))\\
=\sum_{j=1}^{\nu}\frac{\partial F}{\partial z_j}(\textbf{c}_1,x) (h_\lambda^{(i)}(c_{1,j})-\psi_\lambda^{(k)}(c_{1,j}))+\frac{\partial F}{\partial z_{\nu+1}}(\textbf{c}_1, x) (h_\lambda^{(i+1)}(x)-\varphi^{(k)}_\lambda(x))+O(\lambda^{m+2}),
\end{multline*}
as $\lambda\to 0$.  
Summing over $i=0,1,\cdots,k-1$ and using 
 the definition of  $\psi^{(k)}_\lambda(x)$  and  $\varphi^{(k)}_\lambda(x)$ we obtain 
\begin{multline*}
\frac{1}{k} \sum_{i=0}^{k-1}F(h_\lambda^{(i)}(c_{1,1}),\cdots, h_\lambda^{(i)}(c_{1,\nu}), h_\lambda^{(i+1)}(x))\\
=F(\psi^{(k)}_\lambda(c_{1,1}),\cdots, \psi^{(k)}(c_{1,\nu}), \varphi^{(k)}_\lambda(x))+ O(\lambda^{m+2}).
\end{multline*}
Together with (\ref{eqn:arbitraryasyminv1}), this implies the equality in (i).

For (ii), we  use $F(z_1,\cdots, z_\nu)=p_j(z_1, \cdots, z_\nu)$ and argue in a similar way.
\end{proof}

\begin{proof}[Proof of the Main Theorem]
By Proposition~\ref{prop:lift2spectrum}, all eigenvalues of $\mathcal{A}$ are contained in $\overline{\D}$. It remains to prove (2). Let $L=\{v\in \C^\nu: D\mathcal{R}(\textbf{c}_1, v)=0\}$ and let $d$ be the dimension of $L$. By Theorem~\ref{thm:specAtrans}, $L$
has the same dimension as
the eigenspace of $\mathcal{A}$ associated with eigenvalue $1$. Moreover, by Lemma~\ref{lem:arbitraryasyminv}, for each $m\ge 1$ and each $v\in L$, there is a holomorphic motion $h_\lambda(x)$ of $g(P)$ over $\D_\eps$ for some $\eps>0$ which is asymptotic invariant of order $m$ and satisfies  $h_\lambda(c_{1,j})=c_{1,j}+ v(c_{1,j}) \lambda+O(\lambda^2)$ as $\lambda\to 0$. Putting $w_j(\lambda)=h_\lambda(c_{1,j})$, we obtain a holomorphic curve $\lambda\mapsto w(\lambda)=(w_j(\lambda))_{j=1}^\nu$, $\lambda\in \D_\eps$, such that $w_j'(0)=v(c_{1,j})$ and
$$R_j(w(\lambda))=O(\lambda^{m+1})\mbox{ for all }j=1,\dots,\nu.$$

If $d=0$, i.e., $L=\{0\}$, then $D\mathcal{R}(\textbf{c}_1)$ is invertible, so $\mathcal{R}$ is a local diffeomorphism, and for a small neighborhood $W'$ of $\textbf{c}_1$, the set in (\ref{allornoth}) consists of a single point $\textbf{c}_1$.

Now assume $d=\nu$, i.e., $L=\C^\nu$. We claim that $\mathcal{R}(w)\equiv 0.$
Otherwise, there exists $m\ge 1$ such that $R(w)=\sum_{k=m}^\infty P_k(w-\textbf{c}_1)$ in a neighborhood of $\textbf{c}_1$, where $P_k(u)$ is a homogeneous polynomial in $u$ of degree $k$ and $P_m(u)\not\equiv 0$.
 Therefore, there exists $\textbf{v}\in\C^d$ such that $P_m(\lambda\textbf{v})=A \lambda^m$ for some $A\not=0$.  By the argument above, there is holomorphic curve $\lambda\mapsto w(\lambda)$ passing through $\textbf{c}_1$ and tangent to $\textbf{v}$ at $\lambda=0$, such that $|\mathcal R(w(\lambda))|=O(\lambda^{m+1})$. However,
$$\mathcal R(w(\lambda))=P_m(w(\lambda)-\textbf{c}_1))+\sum_{k>m} P_k(w-\textbf{c}_1(\lambda))=A\lambda^m +O(\lambda^{m+1}),$$ a contradiction.


The case $0<d<\nu$ can be done similarly. To be definite, let us assume that
\begin{equation}\label{eqn:Rjinvertible}
\left.\frac{\partial(R_1, R_2,\cdots, R_{d'})}{\partial (w_1, w_2, \cdots, w_{d'})}\right|_{w=\textbf{c}_1}\not=0,
\end{equation}
where $d'=\nu-d$. By the Implicit Function Theorem, there is holomorphic map $\Phi: B\to \C^{d'}$, where $B$ is a neighborhood of
$\textbf{u}_0=(c_{1, d'+1}, c_{1, d'+2}, \cdots, c_{1,\nu})$ in $\C^{d}$
such that
\begin{equation}\label{eqn:Rjled'}
R_j(\Phi(u), u)=0 \mbox{ for all } 1\le j\le d', u\in B,
\end{equation}
and $t=\Phi(u)$ is the only solution of $R_j(t, u)=0, 1\le j\le d'$, in a fixed neighborhood of $(c_{1,1}, c_{1,2},\cdots, c_{1,d})$. It suffices to prove that
$$R_j(\Phi(u), u)=0 \text{ for } u \text{ close to } \textbf{u}_0, d'<j\le \nu.$$
To this end, we only need to show that for any $m\ge 1$, and any $\textbf{e}\in \C^d$, there is a holomorphic curve $u(\lambda)$ in $B$ which passes through $\textbf{u}_0$ at $\lambda=0$, such that $u'(0)=\textbf{e}$ and
\begin{equation}\label{eqn:Rjged'}
R_j(\Phi(u(\lambda)), u(\lambda))=O(\lambda^{m+1}), \text{ as } \lambda\to 0, d'<j\le \nu.
\end{equation}
Indeed, the curve $\tilde{w}(\lambda)=(\Phi(\textbf{u}_0+\lambda \textbf{e}), \textbf{u}_0+\lambda \textbf{e})$ is tangent to $L$ at $\textbf{c}_1$. Thus by the argument in the first paragraph of the proof, there is a curve $w(\lambda)=(t(\lambda), u(\lambda))\in \C^{d'}\times \C^d$, tangent to $\tilde{w}(\lambda)$ at $\lambda=0$ such that
\begin{equation}\label{eqn:Rjall}
R_j(w(\lambda))=O(\lambda^{m+1})\text{ as } \lambda\to 0, 1\le j\le \nu.
\end{equation}
Together with (\ref{eqn:Rjinvertible}) and (\ref{eqn:Rjled'}), this implies that
$$|\Phi(u(\lambda))-t(\lambda)|=O(\lambda^{m+1}).$$
Finally, by (\ref{eqn:Rjall}), we obtain (\ref{eqn:Rjged'}), completing the proof of the Main Theorem.
\end{proof}

\section{Families of the form $f_\lambda(x)=f(x)+\lambda$ and $f_\lambda(x)=\lambda f(x)$}
\label{sec:oneparameter}
In this section we will  apply these techniques to show that one has monotonicity and the transversality properties (\ref{eq:trans}) and (\ref{eq:trans2}) within certain families
of real maps of the form $f_\lambda(x)=f(x)+\lambda$ and $f_\lambda(x)=\lambda \cdot f(x)$
where $x\mapsto f(x)$ has one critical value (and is unimodal - possibly on a subset $\R$) or satisfy symmetries.  There are quite a few
papers giving examples for which one has non-monotonicity for such families,  see for example \cite{Br,Ko,NY,Zd}.
In this section we will prove several theorems which show monotonicity for a fairly wide class
of such families.


In Subsection~\ref{subsec:additive} we show that the methods we developed in the previous section
apply if one has something like a polynomial-like map $f\colon U\to V$  with sufficiently  {\lq}big complex bounds{\rq}. This gives yet another proof for monotonicity
for real families of the form $z^\ell+c$, $c\in \R$ in the setting when $\ell$ is an even integer.
 We also apply this method to a family of maps with a flat critical point
in Subsection~\ref{subsec:flatcritical}. In Subsection~\ref{subsec:multiplicative} we show how to obtain
the lifting property in the setting of one parameter families of the form $f_a(x)=af(x)$ with  $f$ in some rather general class of maps.

\subsection{Families of the form $f_\lambda(x)=f(x)+\lambda$ with a single critical point}\label{subsec:additive}
Let  $f\colon U\to V$ be a map from the class $\mathcal F$   defined in Subsection~\ref{subsec:defclassesEF}.
Consider a marked map $g$ with $g=f+g(0)$  
for some $f\in \mathcal F$ from a finite set
$P$ into itself with $P\supset P_0=\{0\}$, $P\setminus P_0\subset U$. In other words,  $g$ extends to a holomorphic map $g: U_g\to V_g$ where
\begin{itemize}
\item $U_g$ is a bounded open set in $\C$ such that $U_g\supset P\setminus \{0\}$ and $0\in \overline{U}_g$;
\item $V_g$ is a bounded open set in $\C$ such that $c_1:=g(0)\in V_g$;
\item $g: U_g\setminus \{0\}\to V_g\setminus \{c_1\}$ is an unbranched covering;
\item $V_g\supset B(c_1;\diam(U_g)) \supset U_g$.
\end{itemize}
Next define a local holomorphic deformation $(g,G,\textbf{p})_W$ of $g$ as follows: $G_w(z)=g(z)+(w-g(0))$ and $\textbf{p}(w)=0$ for all $w\in W:=\C$.

\begin{theorem}\label{single} Let  $(g,G,\textbf{p})_W$  be as above. 
Then
\begin{enumerate}
\item $(g,G,{\bf p})_W$ satisfies the lifting property;
\item  the spectrum of the operator $\mathcal{A}$ is contained in $\overline{\D}\setminus \{1\}$.
\end{enumerate}
If, in addition,  the robust separation property
$ V_g\supset \overline{B(c_1;\diam (U_g))} \supset U_g$
holds, then the spectral radius of $\mathcal{A}$ is strictly smaller than $1$ and
$$\det (I-\rho \mathcal{A})=\sum_{i=0}^{q-1}\frac{\rho^i}{Dg^i(c_1)}\not=0$$
holds for all $|\rho|\le 1$. In particular, if $g,G$ are real then  $\sum_{i=0}^{q-1}\frac{1}{Dg^i(c_1)}>0$ .
\end{theorem}

\begin{proof}
Let us show that 
$(g, G, \textbf{p})_{\C}$  satisfies the lifting property.
For each domain $\Delta\ni 0$ in $\C$, let $\mathcal{M}_\Delta$ denote the collection of all holomorphic motions $h_\lambda$ of $g(P)$ over $(\Delta,0)$ such that for all $\lambda\in \Delta$,
\begin{equation}
h_\lambda(x)\in U_g\text{ for all }x\in g(P)\setminus \{0\} \mbox{ and }h_\lambda(0)=0.
\end{equation}

{\bf Claim.} Let $\Delta\ni 0$ be a simply connected domain in $\C$. Then any holomorphic motion $h_\lambda$ in $\mathcal{M}_\Delta$ has a lift $\widehat{h}_\lambda$ which is again in the class $\mathcal{M}_\Delta$.

Indeed, $g(P\setminus\{0\})\subset g(U_g\setminus\{0\})=V_g\setminus \{g(0)\}$. So for each $x\in P\setminus\{0\}$, $g(x)\not=g(0)$, thus for any $\lambda\in\Delta$, $$0<|h_\lambda(g(x))-h_\lambda(g(0))|<\diam (U_g),$$
hence by  $V_g\supset B(c_1;\diam(U_g)) \supset U_g$,
$$h_\lambda(g(x))-h_\lambda(g(0))+g(0)\in V_g\setminus \{g(0)\}.$$
Since $g:U_g\setminus\{0\}\to V_g\setminus \{g(0)\}$ is an unbranched covering and $\Delta$ is simply connected, there is a holomorphic function $\lambda\mapsto \widehat{h}_\lambda(x)$, from $\Delta$ to $U_g\setminus\{0\}$, such that $\widehat{h}_0(x)=x$ and
$$g(\widehat{h}_\lambda(x))=h_\lambda(g(x))-h_\lambda(g(0))+g(0),\mbox{ i.e. }\,\,
G_{h_\lambda(g(0))}(\widehat{h}_\lambda(x))=h_\lambda(g(x)).$$
Define $\widehat{h}_\lambda(0)=0$ if $0\in g(P)$. Then $\widehat{h}_\lambda$ is a lift of $h_\lambda$ over $\Delta$.

For any holomorphic motion $h_\lambda$ of $g(P)$ over $(\Lambda,0)$ with $h_\lambda(0)=0$, there is a simply connected sub-domain $\Delta\ni 0$ such that the restriction of $h_\lambda$ to $\Delta$ belongs to the class $\mathcal{M}_{\Delta}$. It follows that $(g,G,\textbf{p})_W$ has the lifting property.

Therefore the assumptions of the Main Theorem
are satisfied.
The operator $\mathcal A$ cannot have an eigenvalue $1$ 
because otherwise for
all parameters $w\in W$ the $G_w$ would have the same dynamics. Hence, (2) in the conclusion of the theorem follows.

If the robust separation property
$ V_g\supset \overline{B(c_1;\diam (U_g))} \supset U_g$ holds, then
Proposition~\ref{prop:perturb2spectrum} applies and therefore the spectral radius of $\mathcal{A}$ is strictly smaller than $1$.
As in Example~\ref{example1} the conclusion follows.
\end{proof}

\begin{coro}\label{cor:zd}
For any even integer $d$, transversality condition (\ref{eq:trans2}) holds and  the topological entropy of $g_c(z)=z^d+c$ depends monotonically on $c\in \R$.
\end{coro}

\subsection{A unimodal family map $f\in \mathcal F$ with a flat critical  point} 
\label{subsec:flatcritical}
Fix $\ell\ge 1$, $b> 2(e\ell)^{1/\ell}$ and consider
$$f_c(x)=\left\{ \begin{array}{rl} b e^{-1/|x|^\ell} +c  &\mbox{ for }x\in \R\setminus \{ 0\}, \\
c & \mbox{ for }x=0 .\end{array} \right.
$$
Note that $\R^+ \ni x\mapsto 2x e^{1/x^\ell}$ has a unique critical point at $x=\ell^{1/\ell}$ corresponding to a minimum value $2(\ell e)^{1/\ell}$.
Therefore the assumption on $b$ implies that $b= 2x e^{1/x^\ell}$ has a unique solution $x=\beta\in (0, \ell^{1/\ell})$.
This implies in particular that  the map $f_{-\beta}$ has the Chebeshev combinatorics: $f_{-\beta}(0)=-\beta$ and $f_{-\beta}(\beta)=\beta$.
Note that
$$D f_{-\beta}(\beta)= be^{-1/\beta^\ell} \frac{\ell }{\beta^{\ell+1}}= \frac{2\ell}{\beta^\ell}>2.$$
Therefore, there exists $x_1>x_0>\beta$ such that
$f_{-\beta}(x_0)=x_1$ and $x_1-\beta> 2(x_0-\beta)$.
Choosing $x_0$ close enough to $\beta$, we have
$$R:=f_0(x_0)=x_1+\beta<b.$$

For a bounded open interval $J\subset \R$, let $D_*(J)$ denote the Euclidean disk with $J$ as a diameter.
This set corresponds to the set of points for which the distance to $J$ w.r.t.
the Poincar\'e metric on $\C_J=\C\setminus (\R\setminus J)$ is at most some $k_0>0$.
Also, let $B^*(x,R)=B(x,R)\setminus \{x\}$ where $B(x,R)$ is the open disc with radius $R$ and centre at $x$.

\begin{lemma}\label{lem:flatext} The map $f_0: (-x_0,0)\cup (0,x_0)\to (0, R)$ extends to an unbranched holomorphic covering map $F_0:U\to B^*(0,R)$, where
$U\subset D_*((-x_0,0))\cup D_*((0,x_0))$.
In particular, $\diam (U)=2x_0 <R$.
\end{lemma}
\begin{proof} Let $\Phi(re^{i\theta})= r^\ell e^{i\ell \theta}$ denote the conformal map from the sector $\{re^{i\theta}: |\theta|<\pi/(2\ell)\}$ onto the right half plane, let $U^+=\Phi^{-1}(D_*((0, x_0^\ell)))$.
Since $\Phi^{-1}\colon \C_{(0,x_0^{\ell})}\to \C_{(0,x_0)}$ is holomorphic, by the Schwarz Lemma $\Phi^{-1}$ contracts
the Poincar\'e metrics on these sets,
and therefore $U^+\subset D_*((0,x_0))$.  Define $U^-=\{-z: z\in U^+\}$, $U=U^+\cup U^-$ and
$$F_0(z)=\left\{\begin{array}{ll}
b e^{-1/\Phi(z)} & \mbox{ if } z\in U^+\\
be^{-1/\Phi(-z)} & \mbox{ if } z\in U^-.
\end{array}
\right.
$$
It is straightforward to check that $F_0$ maps $U^+$ (resp. $U_-$) onto $B^*(0,R)$ as an un-branched covering.
\end{proof}

\begin{coro}\label{coro83}
$F_0\in \mathcal F$. Moreover, if  $c\in U$ then  $F_c=F_0+c$  satisfies the  robust separation property in Theorem~\ref{single}.  
In particular, the kneading sequence of $F_c$ depends monotonically on 
 $c\in [-\beta, \infty)$.
\end{coro}
\begin{proof} Take $U$ as in the previous lemma and take $V=B(c,R)$.
Then $F_c\colon  U \to V\setminus \{c\}$   is an  unbranched covering, and since $\diam(U)<R$
 the robust separation property in Theorem~\ref{single} is satisfied.
\end{proof}


\begin{figure}[h!]
\centering
\includegraphics[width=0.9\textwidth]{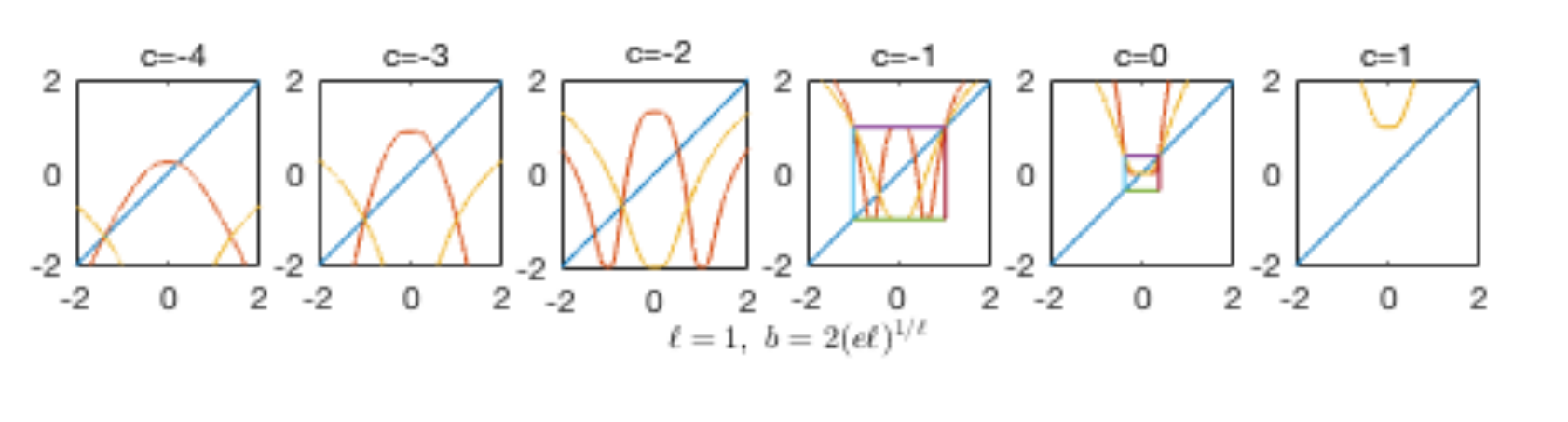}\label{fig1}
\setlength{\abovecaptionskip}{-10pt}
\caption{\small The graphs of $f_c\colon [-2,2]\to [-2,2]$ (in orange) and $f^2_c\colon [-2,2]\to [-2,2]$ (in red) 
for various choices of $c$ when  $b= 2(e\ell)^{1/\ell}$. For $b< 2(e\ell)^{1/\ell}$
there exists no Chebychev parameter $c$, and $b>2(e\ell)^{1/\ell}$ there exists two such
parameters.}
\end{figure}

\begin{figure}[h!]
\centering
\includegraphics[width=0.9\textwidth,height=7cm]{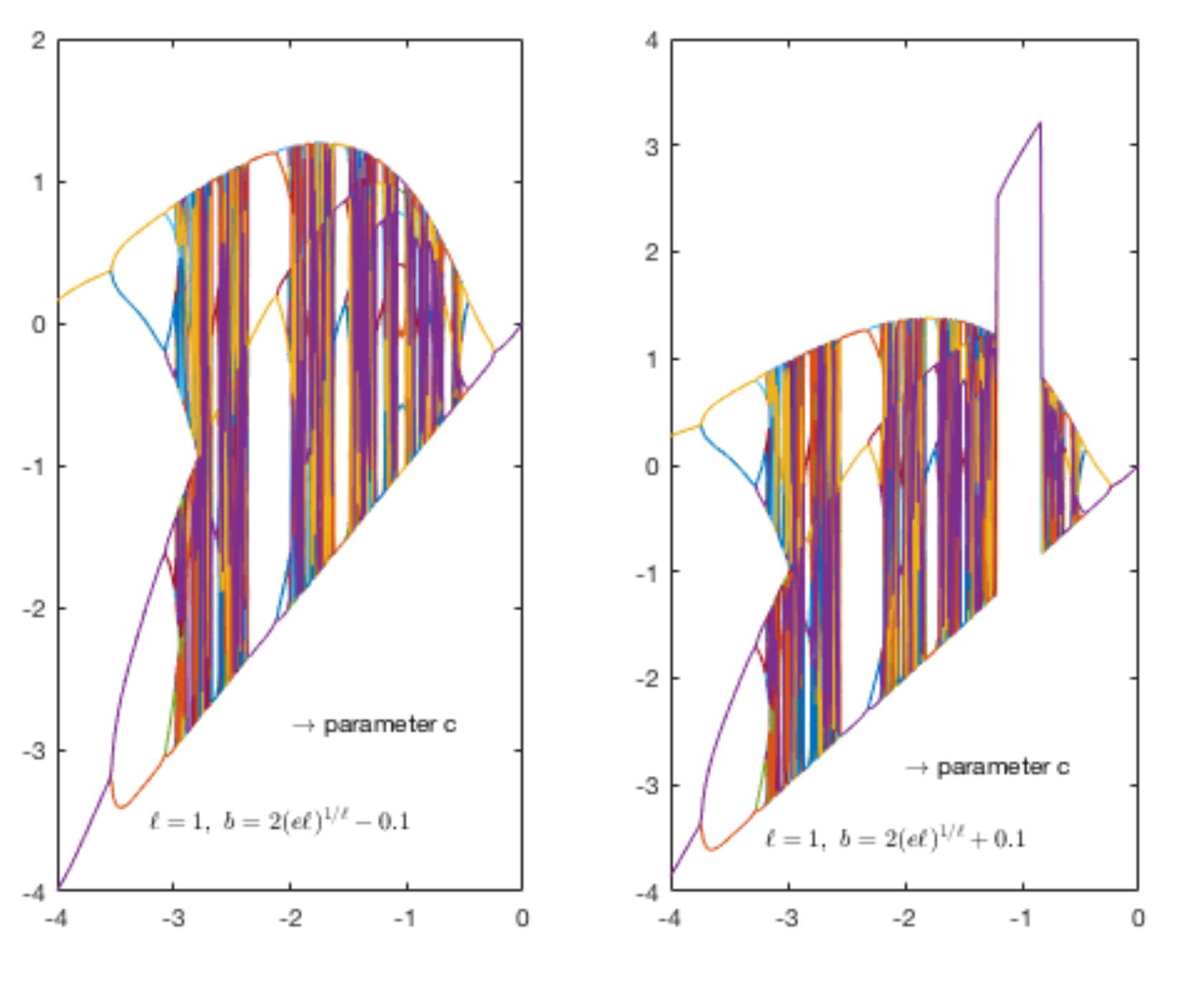}\label{fig2}
\setlength{\abovecaptionskip}{-2pt}
\caption{\small  A bifurcation diagram. Here for $c\in [-4,0]$ (drawn in the horizontal axis),
the last 100 iterates from the set $\{f^k(0)\}_{k=0}^{1000}$ are drawn (in the vertical direction).
Notice that for $b= 2(e\ell)^{1/\ell}+0.1$ the interval $[-\beta_c,\beta_c]$,
where $\beta_c>0$ is
the repelling fixed point of $f_c$ (when it exists), is not invariant for parameters $c\approx -1$
and for these parameter  almost every point is the basin of an attracting fixed point in $\R_+$.
The entropy decreases for $c\in [-\beta,0]$ where for $b= 2(e\ell)^{1/\ell}+0.1$, we have $\beta\approx  -0.831$.}
\end{figure}

\begin{remark}\label{rem:flat} When $b< 2(e\ell)^{1/\ell}$, there no longer exists $c$ so that $f_{c}$ is a Chebychev map,
i.e. so that $f_{c}^2(0)=\beta_c$ where $\beta_c>0$ is a fixed point. Indeed, otherwise
$f_{c}^2(0)=\beta_c$ and therefore since $\beta_c$ is a fixed point and $f_c$ is symmetric,
 $c=f_c(0)=-\beta_c=-f_c(\beta_c)=-f_{-\beta_c}(\beta_c)$ which implies that $\beta_c$ is a solution  of
 $b= 2x e^{1/x^\ell}$.   As we have shown above, this equation has
no solution when  $b< 2(e\ell)^{1/\ell}$. Also note that when $b=2(e\ell)^{1/\ell}$ then  $c=-\beta=-\ell^{1/\ell}$
and
$\frac{d\beta_c}{dc}=1/(1-Df_c(\beta_c))=1/(1-2)=-1$, and therefore the positive oriented transversality of $f_c$ fails at $c=-\beta$.
\end{remark}

\subsection{Families of the form $f_a(x)=af(x)$}\label{subsec:multiplicative}

There are quite a few papers which ask the question:
\begin{quote}For which interval maps $f$, has one monotonicity of the entropy
for the family  $x\mapsto f_a(x)$, $a\in \R$?
\end{quote}
This question is subtle, as the
counter examples to various conjectures show, see \cite{NY, Ko, Br, Zd}.
In this section we will obtain monotonicity and transversality for such families provided
$f$ is contained in the large classes $\mathcal E,\mathcal E_o$  defined in Subsection~\ref{subsec:defclassesEF}. 
For convenience, let us  recapitulate the definitions of the spaces  $\mathcal{E}$ and $\mathcal{E}_o$.
Consider holomorphic maps $f: D\to V$ such that:
\begin{enumerate}
\item[(a)] $D,V$ are open sets which are symmetric w.r.t. the real line so that
$f(D)=V$  
\item[(b)] Let $I=D\cap \R$ then there exists $c>0$ so that $I\cup \{c\}$ is a (finite or infinite) open interval and $0\in \overline I$, $c\in \mbox{int}(\overline{I})$.
Moreover,   $f$ extends continuously to $\overline{I}$,  $f(I)\subset   \R$ and $\lim_{z\in D, z\to 0} f(z)=0$.
\item[(c)]   Let $D_+$ be the component of $D$ which contains $I\cap (c,\infty)$, where $D_+$ might be equal to $D$.
Then $u\in D\setminus \{0\}$ and $v\in D_+\setminus \{0\}$, $v\ne u$,  implies  $u/v\in V$.
\end{enumerate}
\noindent
Let $\mathcal{E}$ be the class of maps which satisfy $(a)$,$(b)$,$(c)$ and assumption $(d)$:
\begin{enumerate}
\item[(d)]   $f\colon D\to V$ has no singular values in $V\setminus \{0,1\}$
and $c>0$ is minimal such that $f$ has a positive local maximum at $c$ and $f(c)=1$.
\end{enumerate}
\noindent
Similarly let  $\mathcal{E}_o$ be the class of maps which satisfy $(a)$,$(b)$,$(c)$ and assumption $(e)$:
\begin{enumerate}
\item[(e)] $f$ is odd,   $f\colon D\to V$ has no singular values in $V\setminus \{0,\pm 1\}$
  and $c>0$ is minimal such that $f$ has a positive local maximum at $c$ and $f(c)=1$.
\end{enumerate}
Here, as usual, we say that $v\in \C$ is a {\em singular value} of a holomorphic map $f: D\to \C$ if it is a critical value,
or an asymptotic value where the latter means the existence of a path $\gamma: [0,1)\to D$ so that $\gamma(t)\to \partial D$ and $f(\gamma(t))\to v$ as $t\uparrow 1¡$. Note that we do not require here that $V\supset D$.

Using qs-rigidity, it was already shown in \cite{RS} that the topological entropy of $\R\ni x \mapsto af(x)$
is monotone $a$, where $f(x)=\sin(x)$ or more generally $f$ is real, unimodal and entire on the complex plane and satisfies a certain sector condition.
Here we strengthen and generalise this result as follows:

\begin{theorem}\label{thm:classE}
Let $f$ be either in $\mathcal{E}$ or in $\mathcal{E}_o$. Assume that the local maximum $c>0$ is periodic for $f_a(x)=af(x)$ where $0<a<b$.
Then the following `positive-oriented' transversality property holds:
\begin{equation} \frac{\frac{d}{d \lambda} f_\lambda^q(c)\left.\right|_{\lambda=a}}{Df_{a}^{q-1}(f_a(c))}>0.\end{equation}
(A similar statement holds when $c$ is pre-periodic for $f_a$.)
In particular, the kneading sequence of the  family $f_a(x): J\to \R$ is
monotone increasing.
\end{theorem}
\begin{proof}
Let $f\in \mathcal{E}\cup \mathcal{E}_o$. Denote $g(x)=af(x)$ and let $V_a=a\cdot V$.
Let $P_0=\{c\}$, $P=\{c_i=g^i(c): i\ge 0\}$. Since $0<a<b$, $f_a$ maps $(0,b)$ into itself, and so $P\subset (0,b)$.
 We may also assume that $g(c)>c$ because otherwise $q=1$ and the result is again trivial.
 By the assumptions, $g$ is a holomorphic map $g: D\to V_a$, $g(P)\subset P$ and $Dg(x)\not=0$ for any $x\in P\setminus P_0$. In particular, $g$ is a real marked map.
For each $w\in W:=\C_*=\C\setminus \{0\}$, $G_w(z):=wf(z)$ is a branched covering from $U:=D\setminus \{0\}$ into $\C$. Define $p(w)\equiv c$. Then $(g, G, p)_W$ is  a local holomorphic deformation of $g$. It suffices to prove that $(g, G, p)_W$ has the lifting property so that the Main Theorem
applies. Indeed, if $1$ is an eigenvalue of $\mathcal A$ then by the Main Theorem
$\{\mathcal{R}(w)=0\}$ is an open set and therefore this critical relation
holds for all parameters,  which clearly is not possible.

Let us first consider the case $f\in \mathcal{E}$. In this case, $w$ is the only critical or singular value of $G_w$. Given a simply connected domain $\Delta\ni 0$ in $\C$, let $\mathcal{M}_\Delta$ denote the collection of all holomorphic motions $h_\lambda$ of $g(P)$ over $(\Delta,0)$ with the following property that for all $\lambda\in \Delta$ we have $h_\lambda(x)\in U$ for all $x\in g(P)\setminus \{c\}$ and $h_\lambda(c)=c$. Given such a holomorphic motion, for each $x\in g(P)$ there is a holomorphic map $\lambda\mapsto \widehat{h}_\lambda(x)$, $\lambda\in\Delta$, with $\widehat{h}_0(x)=x$ and such that
$f(\widehat{h}_\lambda(x))=h_\lambda(g(x))/h_\lambda(g(c))$. Indeed, for $x=c$, take $\widehat{h}_\lambda(x)\equiv c$ and for $x\in g(P)\setminus\{c\}$, we have by  property (c) that $h_\lambda(g(x))/h_\lambda(g(c))\in V\setminus \{0,1\}$. Note that we use here that $g(c)\in D_+$ since $c<g(c)<b$. So the existence of $\widehat{h}_\lambda$ follows from the fact that
$f\colon D\setminus f^{-1}\{0,1\} \to V\setminus \{0,1\}$ is an unbranched covering.  Clearly, $\widehat{h}_\lambda$ is a holomorphic motion in $\mathcal{M}_\Delta$ and it is a lift of $h_\lambda$ over $\Delta$.
It follows that $(g, G, p)_W$ has the lifting property. Indeed, if $h_\lambda$ is a holomorphic motion of $g(P)$ over $(\Lambda, 0)$ for some domain $\Lambda\ni 0$ in $\C$, then we can take a small disk $\Delta\ni 0$ such that the restriction of $h_\lambda$ on $(\Delta, 0)$ is in the class $\mathcal{M}_\Delta$. Therefore, there exists a sequence of holomorphic motions $h_\lambda^{(k)}$ of $g(P)$ over $(\Delta,0)$ such that $h_\lambda^{(0)}=h_\lambda$ and $h_\lambda^{(k+1)}$ is a lift of $h_\lambda^{(k)}$ over $\Delta$ for each $k\ge 0$. If $x=c$ then $h_\lambda^{(k)}(x)\equiv c$ for each $k\ge 1$ while if $x\in g(P)\setminus \{c\}$, $h_\lambda^{(k)}(x)$ avoids values $0$ and $c$. Restricting to a small disk, we conclude by Montel's theorem that $\lambda \mapsto h_\lambda^{(k)}(x)$ is bounded.

The case $f\in\mathcal{E}^o$ is similar. In this case, $G_w$ has two critical or singular values $w$ and $-w$, but it has additional symmetry being an odd function. Given a simply connected domain $\Delta\ni 0$ in $\C$, let $\mathcal{M}_\Delta^o$ denote the collection of all holomorphic motions $h_\lambda$ of $g(P)$ over $(\Delta,0)$ with the following properties: for each $\lambda\in \Delta$,
\begin{itemize}
\item $h_\lambda(x)\in U$ for all $x\in g(P)\setminus \{c\}$ and $h_\lambda(c)=c$;
\item $h_\lambda(x)\not=-h_\lambda(y)$ for $x, y\in g(P)$ and $x\not=y$.
\end{itemize}
Then similar as above, we show that each $h_\lambda$ in $\mathcal{M}_\Delta^o$ has a lift which is again in the class $\mathcal{M}_\Delta^o$. It follows that $(g, G, p)_W$ has the lifting property.
\end{proof}

Let us now show that this theorem applies to a unimodal family with a flat critical point:

\begin{lemma} \label{lemflat} The unimodal map $f\colon [0,1]\to [0,1]$ defined by
\begin{equation}
f(x)=\exp(2^\ell) \left( -\exp(-1/|x-1/2|^\ell) + \exp(-2^\ell) \right)\label{mflat}
\end{equation}
has a holomorphic unbranched extension  $f\colon D_-\cup D_+ \to B^*(1,1)$.  Here  $D_-,D_+$ are domains with $D_-\subset B_*(0,1/2)$, $D_+\subset D_*(1/2,1)$ and $B^*(x,r)$ is the ball with radius $r$ centered and punctured at $x$.
There exists $\ell_0>0$ so that for all $\ell>\ell_0$, one can
slightly enlarge $V$, $D_-,D_+$ and obtain a map in $\mathcal{E}$.
\end{lemma}
\begin{proof} Let us consider $f$ as a composition of a number of maps.
First  $z\mapsto  (1/2-z)^\ell$ and $z\mapsto (z-1/2)^\ell$ map some domains $U_-,U_+$ which
are contained in $D_*(0,1/2)$ and $D_*(1/2,1)$ onto  $D_*(0,1/2^\ell)$ when $\ell\ge 1$.
Next $z\mapsto -1/z$ maps $D_*(0,1/2^\ell)$ onto a half-plane $Re(z)\le -2^\ell$.
Then $z\mapsto \exp(z)$ maps this half-plane onto the punctured disc $B^*(0,\exp(-2^\ell))$ centered at $0$ and with
radius $\exp(-2^\ell)$ (and with a puncture at $0$).
Applying the translation $z\mapsto \exp(-2^\ell)$ to this punctured disc we obtain the  punctured disc centered at
$B^*(\exp(-2^\ell),\exp(-2^\ell))$.
Then multiplying this disc by $\exp(2^\ell)$ shows that $f$ maps
$U_-,U_+$ onto $B^*(1,1)$.  (Note that this final punctured disc touches the imaginary axis.)
Since $0$ is a repelling fixed point of $f$ with multiplier $>2$, and $U_-,U_+$ are close to the intervals $(0,1/2)$ and $(1/2,1)$
when $\ell$ is large,  we can enlarge the domain and range, and obtain a map as in (a)-(d).
\end{proof}

\begin{figure}[h!]
\centering
\includegraphics[width=0.9\textwidth,height=6cm]{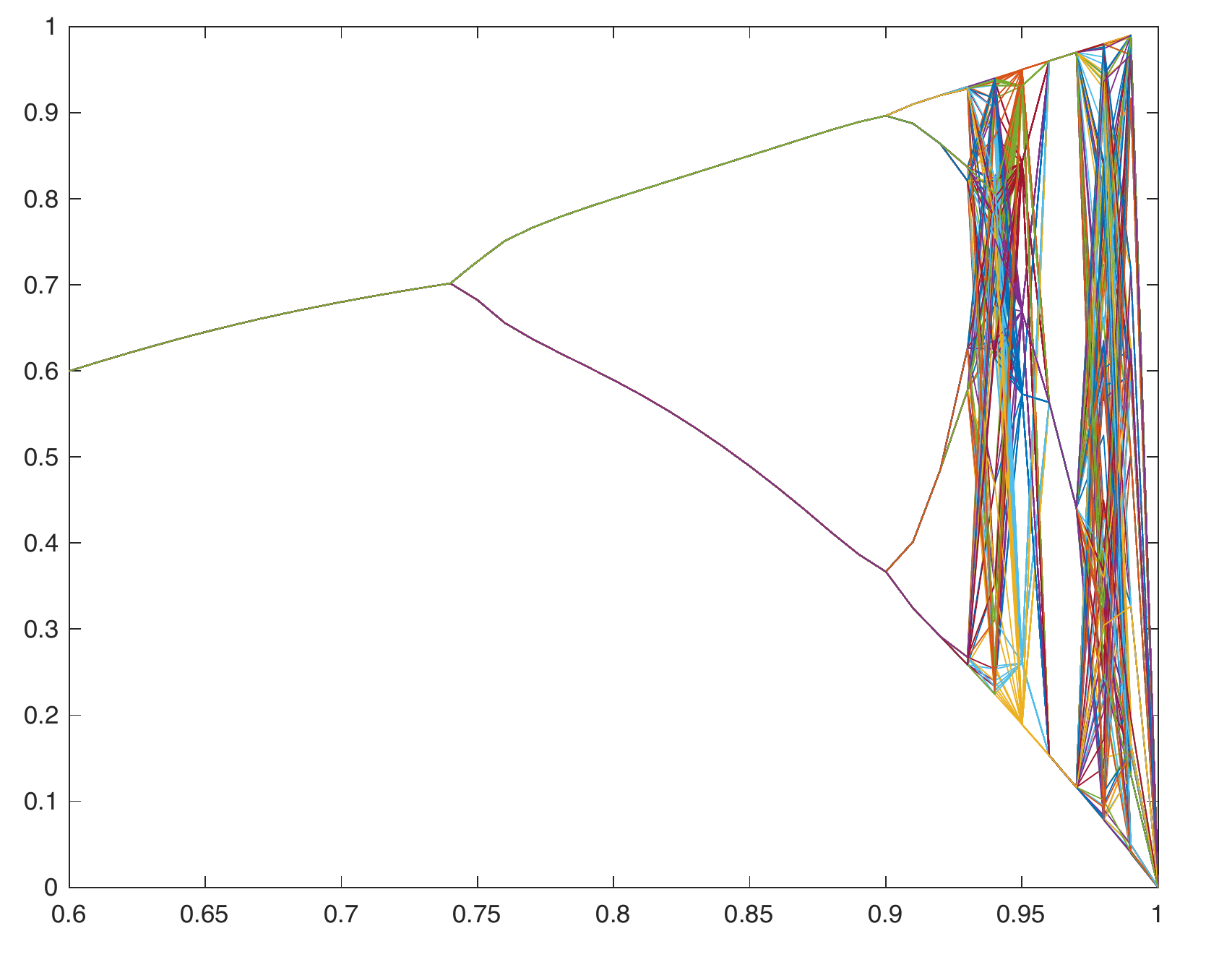}\label{fig3}
\setlength{\abovecaptionskip}{-2pt}
\caption{\small  A bifurcation diagram for the family $f_\lambda(x)=\lambda f(x)$, $\lambda\in [0,1]$
where $f$ is as formula (\ref{mflat}). Note that this unimodal family with a flat critical point is monotone.}
\end{figure}

\section{Application to families with one free critical point}
\label{sec:parameter}

Let us apply the method along a curve in  parameter space corresponding
to where some $\nu$-parameter family of maps $G_w$ has  $\nu-1$ critical relationships.

Let us say that $(G,\textbf{p})_W$ is a {\em partially marked family of maps} if
\begin{enumerate}
\item $W$ is an open connected subset of $\C^\nu$ and $U$ is an open subset of $\C$;
\item $\textbf{p}=(p_1,p_2,\ldots, p_\nu):W\to \C^\nu$ is a holomorphic map and
so that all coordinates of  $\textbf{p}(w)$ are distinct. Let $P_{0,w}=\{p_1(w),\dots,p_\nu(w)\}$.
\item $G: (w,z)\in W\times U \mapsto G_w(z)\in \C$ is a holomorphic map
and $DG_w(z)\ne 0$ for each $w\in W$ and $z\in U$.
\item associated to each $j=1,\dots,\nu-1$ there exists a positive integer $q_j$ so that $G^k_w(w_j)\in U$
for $k=1,\dots,q_j-1$ and $w=(w_1,\cdots,w_{\nu-1},w_\nu)\in W$.
\end{enumerate}

Assume that the family map $G_w$ is real. That is, assume that the properties
defined below Definition~\ref{def:unfoldtrans} hold:
for any $w=(w_1, w_2, \ldots, w_\nu)\in W$, $z\in U$ and $j=1,2,\ldots,\nu$, we have
$\overline{w}= (\overline{w_1}, \overline{w_2},\ldots, \overline{w_\nu})\in W$, $\overline{z}\in U$, and
$$G_{\overline{w}}(\overline{z})=\overline{G_w(z)}, \text{ and } p_j(\overline{w})=\overline{p_j(w)}.$$

Choose for each $j=1,\dots,\nu-1$ either
$\mu(j)\in \{1,2,\ldots,\nu\}$ or  $1\le l_j<q_j$. Given this choice,
let $L$ be the set of $w=(w_1,\cdots,w_\nu)\in W$ for which the following hold:
\begin{enumerate}
\setcounter{enumi}{4}
\item  if $\mu(j)$ is defined then $G_w^{q_j-1}(w_j)=p_{\mu(j)}(w)$ and $G_w^k(w_j)\not\in P_{0,w}$ for each $1\le k<q_j$;
\item if $l_j$ is defined then
$G_w^{q_j-1}(w_j)=G_w^{l_j-1}(w_j)$ and $G_w^k(w_j)\not\in P_{0,w}$ for all $0\le k\le q_j-1$.
\end{enumerate}
Relabelling these points $w_1,\dots,w_{\nu-1}$, we assume that there is $r$ such that
the first alternative happens for all $1\le j\le r$ and the second alternative happens for $r<j\le \nu-1$.

\begin{remark}
So for each $w\in L$, $G_w$ has $\nu-1$ critical relations which start with $p_1(w),\dots,p_{\nu-1}(w)$.
Hence  the terminology of partially marked family of maps.  \end{remark}

Define for $w\in W$, for $1\le j\le r$,
\begin{equation}R_j(w)=G_{w}^{q_j-1}(w_j)-p_{\mu(j)}(w)\label{R1-r}
\end{equation}
and for $r<j\le \nu-1$,
\begin{equation}R_j(w)=G_{w}^{q_j-1}(w_j)-G_{w}^{l_j-1}(w_j),\label{Rr-nu}
\end{equation}
where $w=(w_j)_{j=1}^\nu$.
Then $L$ is precisely the set
$$L=\{w\in W; \; R_1(w)=\dots=R_{\nu-1}(w)=0\}.$$

Let $L_*$ be a maximal connected subset of $L\cap \R^\nu$ such that for each $w\in L_*$, the $\nu\times (\nu-1)$ matrix
\begin{equation}
V_w:=\left[\dfrac{1}{DG_w^{q_1-1}(w_1)}\nabla R_1(w),\dots, \dfrac{1}{DG_w^{q_{\nu-1}-1}(w_{\nu-1})}\nabla R_{\nu-1}(w)\right]\label{eq:maxrank} \end{equation}
has rank $\nu-1$. Here $\nabla R_i(w)$ is the gradient of $R_i$.
By the implicit function theorem, $L_*$ is a real analytic curve.

Now, let us assume that for some ${\bf c}_1=(c_{1,1},\dots,c_{1,\nu}) \in L_*$, $G_{{\bf c}_1}$ has an additional  critical relation starting with $p_\nu(w)$, i.e., $g=G_{{\bf c}_1}$ is a {\em marked map} and $G$ extends to a local holomorphic deformation of $g$: there is a neighborhood $W_{{\bf c}_1}\subset W$ of  ${\bf c}_1$, an open set $U_{{\bf c}_1}\supset U$ such that $G: W_{{\bf c}_1}\times U\to W_{{\bf c}_1}\times \C$ extends to a holomorphic map
$G: (w,z)\in W_{{\bf c}_1}\times U_{{\bf c}_1}\mapsto (w,G_w(z))\in W_{{\bf c}_1}\times \C$
and $DG_w|U_{{\bf c}_1}\ne 0$ for each $w\in W_{{\bf c}_1}$ and
there exists either $\mu(\nu)\in \{1,2,\ldots,\nu\}$ or  $l_\nu<q_\nu$
so that
either
\begin{itemize}
\item  if $\mu(\nu)$ is defined then $g^{q_\nu}(p_\nu({\bf c}_1))=p_{\mu(\nu)}({\bf c}_1)$, $g^k(c_{1,\nu})\not\in P_{0,{\bf c}_1}$ for each $1\le k<q_\nu$
then define
$$R_\nu(w)=G_{w}^{q_\nu-1}(w_\nu)-p_{\mu(\nu)}(w)$$
\item if $l_\nu$ exists then
$g_{{\bf c}_1}^{q_\nu}(p_\nu({\bf c}_1))=g^{l_\nu}(p_\nu({\bf c}_1))$ and $g^k(w_\nu)\not\in P_{0,{\bf c}_1}$ for all $1\le k\le q_\nu$,
then we define
$$R_\nu(w)=G_{w}^{q_\nu-1}(w_\nu)-G_{w}^{l_\nu-1}(w_\nu).$$
\end{itemize}
Notice that $R_\nu$ is only defined in a small neighbourhood $W_{{\bf c}_1}$ of ${\bf c}_1$.

The following theorem gives a condition
 implying that along the curve $L_*$ all bifurcations
are in the same direction.

\begin{theorem}\label{thm:curves}
For each $w\in L_*$, define $E_w\in T_w\C^\nu$ to be the unique unit vector in $\R^\nu$ orthogonal to the range of the matrix $V_w$
and so that
$$\det \left[\dfrac{1}{DG_w^{q_1-1}(w_1)}\nabla R_1(w),\dots, \dfrac{1}{DG_w^{q_{\nu-1}-1}(w_{\nu-1})}\nabla R_{\nu-1}(w), E_w\right]>0.$$
Then
\begin{itemize}
\item $E_w$ is a tangent vector to $L_*$ at $w$ and $L_*\ni w\mapsto E_w$ is real analytic.
In particular, $E_w$ defines an orientation on the entire curve $L_*$ which we will call {\lq}positive{\rq}.
\item If for some ${\bf c}_1\in  L_*$ the corresponding map $g=G_{{\bf c}_1}$
is a marked map as above  
and the positively oriented transversality property (\ref{eq:trans2}) holds for the local holomorphic deformation
$(g,G,p)_{W_{{\bf c}_1}}$, then
$$\dfrac{1}{Dg^{q_\nu-1}(c_{1,\nu}) } \nabla_{E}R_\nu ({\bf c}_1)>0$$
where  $Dg^{q_\nu-1}(c_{1,\nu})$ is the spatial derivative, and $\nabla_{E}R_\nu({\bf c}_1)$ is the derivative in the direction of
the tangent vector $E=E_{{\bf c}_1}$ of $L_*$ at ${\bf c}_1$.
\end{itemize}
\end{theorem}

\begin{proof}
Let $A=[A_1,\dots,A_\nu]$ be the $\nu\times \nu$ matrix with the first $(\nu-1)$-columns equal to the columns of
$V_{{\bf c}_1}$ and the last column equal to $\dfrac{1}{DG_{{\bf c}_1}^{q_{\nu}-1}(c_{1,\nu})}\nabla R_{\nu}({{\bf c}_1})$.
Note that the determinant of this matrix is positive by the positive oriented transversality condition
(\ref{eq:trans2}).
There exists $\lambda$ so that  $A_\nu=\lambda E_{{\bf c}_1} +  v$ where $v$ is
in the range of the matrix $V_{{\bf c}_1}$.
So $\det(A_1,\dots,A_{\nu-1},A_\nu)=\det(A_1,\dots,A_{\nu-1},\lambda E_{{\bf c}_1} +  v)=\lambda \det(A_1,\dots,A_{n-1},E_{{\bf c}_1})$.
Since the first and the last determinants are positive, we have $\lambda>0$ and so $A_\nu \cdot E_{{\bf c}_1}>0$.
This is precisely the expression claimed to be positive in the theorem.
\end{proof}

\begin{remark}\label{rem:ent1}
Applying this theorem to the setting of a family of globally defined real analytic maps,
we obtain monotonicity of entropy along such curves $L_*$. This holds because
 the topological entropy is equal to the growth rate
of the number of laps  for continuous piecewise monotone
interval or circle maps, \cite{MS}.
\end{remark}

%

\subsection{Application to {\lq}bone{\rq} curves in the space of real cubic maps}

For every $(a,b)\in \Sigma:=\R^2\setminus \{a=0\}$, let $f_{a,b}(x)=x^3-3a^2x+b$. Then $f_{a,b}$ has critical points $\pm a$.
It is clear that for any $(a,b)\in \C^2\setminus \{a=0\}$, $f_{a,b}$ is locally parametrized by its (different) critical values $w_{1,2}=\pm 2a^3+b$.
For $q>0$ consider a connected component $L_{q}$ of the set
$$\{(a,b)\in \Sigma: f_{a,b}^q(a)\in \{\pm a\}, f_{a,b}^k(a)\notin \{\pm a\}, k=1,\cdots,q-1\}.$$
By \cite{LSvS}, the corresponding $2\times 1$ matrix  (\ref{eq:maxrank})    has rank one
and hence $L_{q}=L_{q*}$ is a simple smooth curve. In fact, the positively oriented transversality property holds for  any critically-finite $f_{a,b}$;
this follows similar to the proof of Theorem~\ref{single} of the next Section (see \cite{LSvScompanion} for a general result though).
By Theorem \ref{thm:curves} we have a positive orientation on $L_q=L_q^*$ and the entropy
increases or decreases along this curve as mentioned in Remark~\ref{rem:ent1}.


For $q>0$ consider a connected component $\Gamma_{q}$ of the set
$$\{(a,b)\in \Sigma: f_{a,b}^q(a)=a\}$$
which was called a {\em bone} in \cite{MT}.
The next theorem proves a crucial property of this set, which was derived in \cite{MT}
using global considerations (including Thurston rigidity for postcriticallly finite maps). Here we will derive
this property from positive transversality.

\begin{figure}
\centering
\begin{subfigure}{.5\textwidth}
  \centering
  \includegraphics[width=0.9\textwidth]{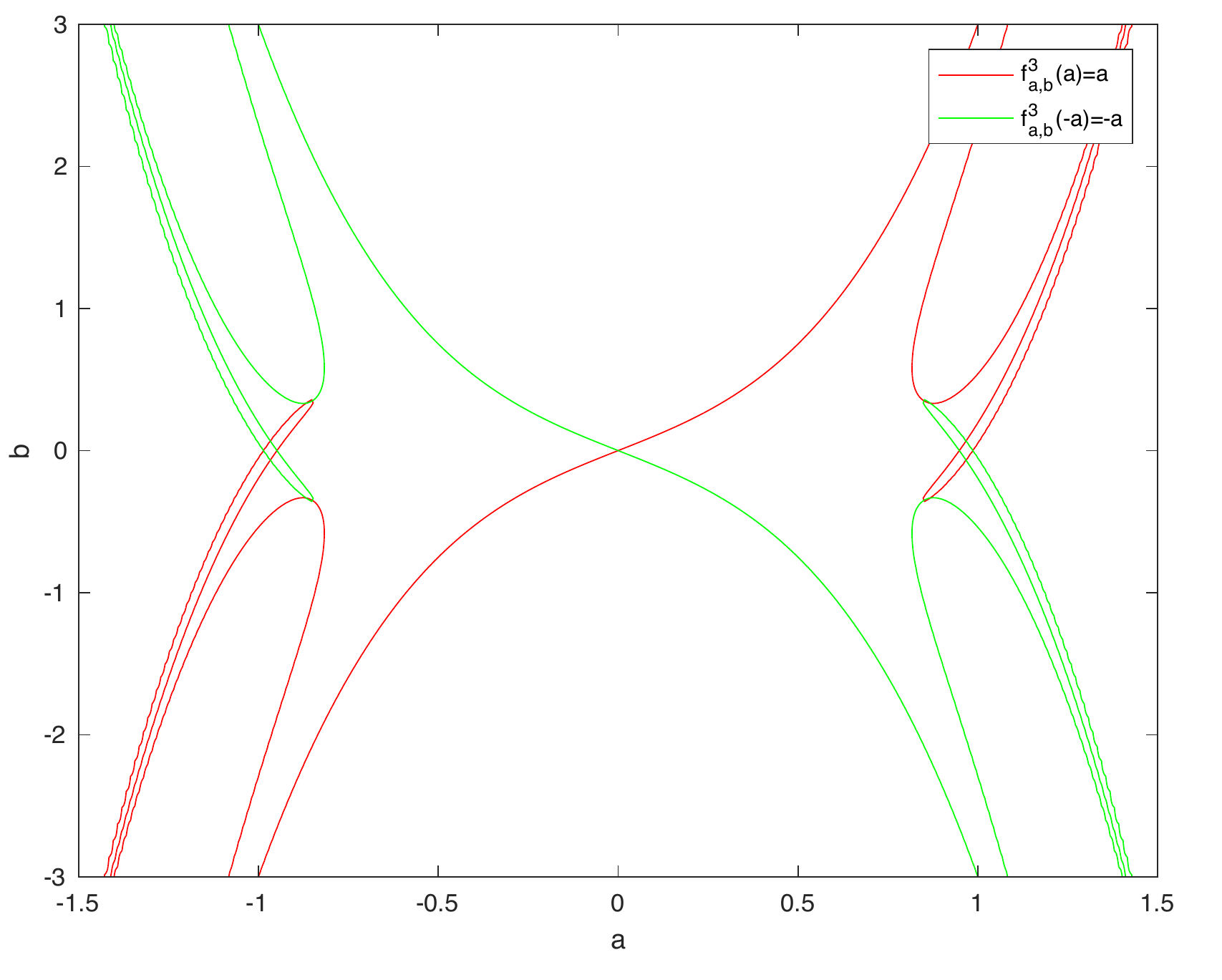}\label{fig4a}
\end{subfigure}%
\begin{subfigure}{.5\textwidth}
  \centering
  \includegraphics[width=0.9\textwidth]{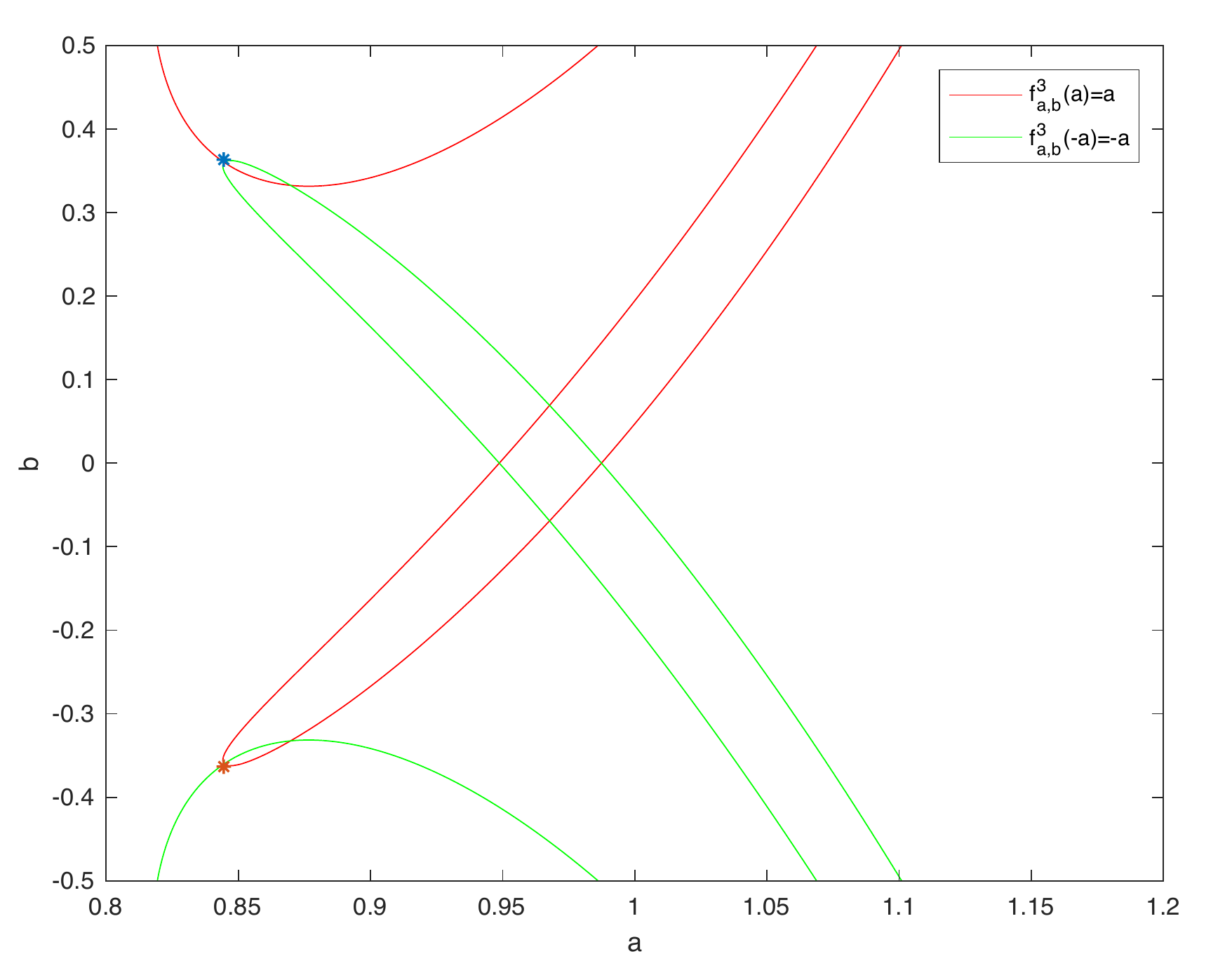}\label{fig4b}
 \label{fig:renormalization}
\end{subfigure}
\caption{\small The sets $\Gamma_3(\pm a)=\{(a,b); f_{a,b}^3(\pm a)=\pm a\}$ for the cubic family $f_{a,b}=x^3 - 3a^2 x + b$ (with critical points $\pm a$). The two curves in the left panel crossing at $(0,0)$ correspond
to the sets where one of the critical points of $f_{a,b}$ is a fixed point, i.e.  where $\{f_{a,b}(\pm a)=\pm a\}$. The two points marked with $*$
in the magnification on the right are where the two critical points $\pm a$ lie on the same orbit. By Theorem~\ref{thm:bones}, the topological entropy is monotone on the components of $\Gamma_3(\pm a)$ minus these points. This information was already obtained by Milnor and Tresser in
\cite{MT} but here we derive it from the local methods derived in this paper.}
\label{fig:blahblah}
\end{figure}

\begin{theorem}[Properties of bones]
\label{thm:bones}
Assume that for some $(\tilde a,\tilde b) \in \Gamma_q$ the integer $q>0$ is minimal so that $f_{\tilde a,\tilde b}^q(\tilde a)=\tilde a$.
Then for all $(a,b) \in \Gamma_q$ one has that  $f_{a,b}^i(a)\ne  a$ for all $0<i<q$. Moreover,
\begin{enumerate}
\item there exists at most one $(a_*, b_*) \in \Gamma_q$ so that $f_{a_*,b_*}^i(a_*)=-a_*$ for some $0\le  i \le  q$.
\item the kneading sequence of $f_{a,b}$ is monotone on each of the components of $\Gamma_q\setminus \{(a_*,b_*)\}$;
more precisely, it  is non-decreasing
on one component and non-increasing on the other component. 
\end{enumerate}
\end{theorem}
\begin{proof}
That  $f_{a,b}^i(a)\ne  a$ for all $0<i<q$, $(a,b)\in \Gamma_q$  follows from the implicit function theorem because the multiplier of this
$q$-periodic orbit is not equal to $1$.  It is well known that
$\Gamma_q$ is a smooth curve (this also follows for
example  from \cite{LSvS}).
The curve $\Gamma_q$ is a component of the zero set of
$$\tilde R(a,b)=f^q_{a,b}(a)-a.$$
As remarked, the critical values $w=(w_1,w_2)$ are local parameters along the curve $\Gamma_q$
and we define a 
direction on the curve  $\Gamma_q$ by the tangent vector
$V_{(a,b)}=(-\dfrac{\partial \tilde R}{\partial w_2},\dfrac{\partial \tilde R}{\partial w_1})$.

Assume that for some $(a_*,b_*)\in \Gamma_q$
the orbit of $a_*$ contains the other critical point, i.e. assume that  $f^i_{a_*,b_*}(a_*)=-a_*$ for some $0<i<q$.
The idea of the proof below is as follows. We will show that as the point $(a,b)\in \Gamma_q$ passes through
$(a_*,b_*)$, the point $-a$ crosses $f^{i}_{a,b}(a)$  in a  direction which depends only on the sign of
$\Delta_i(a_*,b_*)$ where
\begin{equation} \Delta_i(a,b):=   \prod_{1\le k<q, k\not =i} Df_{a,b}(f_{a,b}^k (a)) 
\label{eq:neg33} \end{equation}
so $$\Delta_i(a_*, b_*)=Df_{a_*,b_*}^{i-1}(f_{a_*,b_*}(a_*)) Df_{a_*,b_*}^{q-i-1}(f_{a_*,b_*}(-a_*)).$$

Now as $(a,b)$ moves further along $\Gamma_q$, on the one hand (for the same reason)
$-a$ cannot cross $f^{i}_{a,b}(a)$  in the opposite  direction, and on the other hand
 $-a$ cannot cross a neighbour $f^j_{a,b}(a)$ of $f^i_{a,b}(a)$ at some other
 $(a_\bullet,b_\bullet)\in \Gamma_q$ because $\Delta_i(a_*,b_*)$ and $\Delta_j(a_\bullet,b_\bullet)$
 have opposite signs. Let us explain this in more detail.

Since the map $f_{a_*,b_*}$ is critically finite, one has positive transversality at this parameter.
More precisely, define
$$R_1(a,b)=f^i_{a,b}(a)-(-a), R_2(a,b)=f^{q-i}_{a,b}(-a)-a.$$
Now observe that $$\left( \dfrac{\partial  R_2}{\partial w_1},  \dfrac{\partial  R_2}{\partial w_2} \right) _{(a_*,b_*)}=
\left( \dfrac{\partial  \tilde R}{\partial w_1},  \dfrac{\partial \tilde R}{\partial w_2} \right) _{(a_*,b_*)}.$$
Hence the positive transversality condition can be written as
$$\dfrac{1}{\Delta_i(a_*,b_*)} 
\det \left( \begin{array}{cc} \dfrac{\partial  R_1}{\partial w_1}
& \dfrac{\partial  R_1}{\partial w_2} \\  & \\ \dfrac{\partial  \tilde R}{\partial w_2} & \dfrac{\partial  \tilde R}{\partial w_1}
\end{array}\right) _{(a_*,b_*)} >0
$$
As in the proof of  Theorem~\ref{thm:curves} we obtain
\begin{equation}
\dfrac{1}{\Delta_i(a_*,b_*)}  D_{V_{(a_*,b_*)}} R_1(a_*,b_*) >0\label{eq:p} \end{equation}
where $D_{V_{(a_*,b_*)}}$ stands for the directional derivative of $R_1$ in the direction $V_{(a_*,b_*)}$.
To be definite, let us consider the case that
\begin{equation}\Delta_i(a_*,b_*)  <0. \label{eq:neg} \end{equation}
This implies that
\begin{equation} D_{V_{(a_*,b_*)}} R_1(a_*,b_*) <0\end{equation}
and so the derivative of
\begin{equation}\Gamma_q\ni (a,b) \mapsto f^i_{a,b}(a)-(-a)\mbox{ is  negative at }(a_*,b_*).\label{eq:de1} \end{equation}
By contradiction, assume that there exists another parameter $(a_\bullet,b_\bullet)\in \Gamma_q$ which is the nearest to the right of $(a_*,b_*)$
 for which there exists $0<j<q$ so that   $f^{j}_{a_\bullet,b_\bullet}(a_\bullet)=-a_\bullet$.
In what follows we use that the ordering of the points $a,\dots,f^{q-1}_{a,b}(a)$ in $\R$ does not change along the curve $\Gamma_q$.
Let   $\Gamma_q^\bullet$ be the open arc between $(a_*,b_*)$ and $(a_\bullet,b_\bullet)$.
Notice that  because of  (\ref{eq:neg})
\begin{equation} \Delta_i(a,b)<0\label{eq:neg3} \end{equation}
for all $(a,b)\in \Gamma_q^\bullet$.
Observe that $j\ne i$ because when $j=i$  then  (\ref{eq:neg3})  holds on the closure of $\Gamma_q^\bullet$ and so
$D_{V_{(a_\bullet,b_\bullet)}} R_1(a_\bullet,b_\bullet) <0$. Therefore
(\ref{eq:de1}) also holds at $(a_\bullet,b_\bullet)$ which is clearly a contradiction.
Therefore, $j\ne i$ and
\begin{equation}
f_{a,b}^i(a)<-a<f_{a,b}^j(a)\label{eq:ineq} \end{equation} along the open arc  $\Gamma_q^\bullet$
and there are no points of the orbit of $a$ between $f^i_{a,b}(a),f^j_{a,b}(a)$.
The sign of
\begin{equation}\Delta_j(a,b):= Df_{a,b}^{j-1}(f_{a,b}(a))Df_{a,b}^{q-j-1}(f_{a,b}(-a)) \label{eq:neg2} \end{equation}
is constant for $(a,b)$ near $(a_\bullet,b_\bullet)$. Therefore for $(a,b)\in \Gamma_q^\bullet$,
$$\sign\Delta_j(a_\bullet,b_\bullet)=\sign \Delta_j(a,b)=\sign \left[ \Delta_i (a,b) \dfrac{Df_{a,b}(f^i_{a,b}(a))}{Df_{a,b}(f^j_{a,b}(a))}\right].$$
Because of (\ref{eq:neg3}) we therefore have
$$\sign\Delta_j(a_\bullet,b_\bullet)= - \sign \left[ \dfrac{Df_{a,b}(f^i_{a,b}(a))}{Df_{a,b}(f^j_{a,b}(a))}\right].$$
The key point is that the sign of the ratio in the r.h.s. of this expression is negative because
$-a$ is a folding critical point and because of (\ref{eq:ineq}).
It follows that $\Delta_j(a_\bullet,b_\bullet)>0$ and so arguing as before the derivative of
\begin{equation}\Gamma_q\ni (a,b) \mapsto f^j_{a,b}(a)-(-a)\mbox{ is positive at }(a_\bullet,b_\bullet).\label{eq:de2} \end{equation}
But by (\ref{eq:ineq}) we have that
$ f^j_{a,b}(a)-(-a) >0$ on $\Gamma_q^\bullet$. This and (\ref{eq:de2})  imply that $ f^j_{a_\bullet,b_\bullet}(a_\bullet)-(-a_\bullet)>0$
which is a contradiction.

The 2nd assertion follows immediately from Theorem~\ref{thm:curves} and Remark~\ref{rem:ent1}.
\end{proof}

\begin{remark} The proof of the previous theorem can also be applied to the setting of polynomials of higher degrees.
\end{remark}

\appendix

\section{The family $f_c(x)=|x|^{\ell_\pm}+c$ with $\ell_{\pm}>1$ large}\label{sec:finiteorder}
In this section we obtain monotonicity for unimodal (not necessary symmetric!) maps
in the presence of critical points of large non-integer order, but only
if not too many points in the critical orbit are in the orientation reversing branch.
\subsection{Unimodal family with high degrees}

\begin{theorem}\label{thm:finiteorder2}
Fix real numbers $\ell_-,\ell_+\ge 1$ and consider the family of unimodal maps $f_c=f_{c,\ell_-,\ell_+}$ where
$$f_c(x)=\left\{\begin{array}{ll}
|x|^{\ell_-}+c & \mbox{ if } x\le 0\\
|x|^{\ell_+}+c & \mbox{ if } x\ge 0.
\end{array}
\right.
$$
For any integer $L\ge 1$ there exists $\ell_0>1$
so that for any $q\ge 1$ and any
periodic kneading sequence   $\bold i=i_1i_2\cdots\in \{-1,0,1\}^{\Z^+}$ of period $q$
so that
$$\#\{0\le j< q ; i_j =-1 \}\le L,$$
and any pair $\ell_-,\ell_+\ge \ell_0$ there is at most one $c\in\R$ for which the kneading sequence
of $f_c$ is equal to $\bold i$. Moreover,
\begin{equation}\label{fiortrans2}
\sum_{n=0}^{q-1} \frac{1}{Df_c^n(c)}>0.
\end{equation}
\end{theorem}

{\bf Notations.} As usual, for any three distinct point $o, a, b\in\C$, let $\angle aob$ denote the angle in $[0,\pi]$ which is formed by the rays $oa$ and $ob$. We shall often use the following obvious relation: for any distinct four points $o,a, b, c$,
$$\angle aob +\angle boc \ge \angle aoc.$$
For $\theta\in (0, \pi)$, let
$$D_{\theta}=\{z\in\C\setminus\{0,1\}: \angle 0z1>\pi-\theta\}$$ and let
$$S_\theta=\{re^{it}: t\in (-\theta, \theta)\}.$$ For $0<t<1$, we shall only consider $z^t$ in the case $z\not\in (-\infty, 0)$ and $z^t$ is understood as the holomorphic branch with $1^t=1$.

Let us fix a map $f=f_{c,\ell_-,\ell_+}$ with a periodic critical point of period $q$ and let $P=\{f^n(0): n\ge 0\}$. So $P$ is a forward invaraint finite set.
Denote
$$\ell=\min\{\ell_-,\ell_+\}.$$
\begin{definition}
A holomorphic motion $h_\lambda$ of $P$ over $(\Omega,0)$, is called ${\theta}$-regular if
\begin{enumerate}
\item[(A1).] For $a\in P$, $$h_{\lambda}(a)\in S_{4\theta/\ell} \text{ , if } a>0$$
and
$$h_{\lambda}(a)\in -S_{4\theta/\ell} \text{ , if } a<0;$$
\item[(A2).] For $a, b\in P$, $|a|>|b|>0$ and $ab>0$, $$\frac{h_{\lambda}(b)}{h_{\lambda}(a)}\in D_{\theta}.$$
\end{enumerate}
\end{definition}

Given a $\theta$-regular holomorphic motion $h_{\lambda}$ of $P$ over $\Omega$, with $\theta\in (0, \pi)$, one can define another holomorphic motion $\tilde{h}_{\lambda}$ of $P$ over the same domain $\Omega$ as follows: $\tilde{h}_{\lambda}(0)=0$; for $a\in P$ with $a>0$,
$$\tilde{h}_\lambda(a)=(h_{\lambda}(f(a))-h_{\lambda}(f(0)))^{1/\ell_+};$$ for $a\in P$ with $a<0$, define
$$\tilde{h}_\lambda(a)=-(h_{\lambda}(f(a))-h_{\lambda}(f(0)))^{1/\ell_-}.$$

 The new holomorphic motion is called the {\em lift} of $h_{\lambda}$ which clearly satisfies the condition (A1), but not necessarily (A2) in general.

\begin{mainlemma} There is $\ell_0$ depending only on the number $L$  such that for any $\ell\ge \ell_0$ and each $\theta$ small enough, the following holds:
If $\#\{0\le j< q ; i_j =-1 \}\le L$ and
if  a $\theta$-regular motion can be successively lifted $q-1$ times and all these successive lifts are $\theta$-regular, then the $q$-th lift of the holomorphic motion is $\theta/2$-regular.
\end{mainlemma}

\begin{proof}[Proof of Theorem~\ref{thm:finiteorder2}]
Given $L$, choose $\ell_0$ as in the Main Lemma. It is enough to prove (\ref{fiortrans2}) provided $\ell\ge \ell_0$. Consider a local holomorphic deformation $(f_c,f_w,\textbf{p})_W$ where $W\subset \C$ is a small neighbourhood of $c$, $f_w=f_c+(w-c)$ and $\textbf{p}=0$.
Let $h_\lambda$ be a holomorphic motion of $P$ over $(\Delta,0)$. Let us fix $\theta>0$ small enough.
Restricting $h_\lambda$ to a smaller domain $\Delta_\eps$, we may assume that $h_{\lambda}$ is $\theta$-regular and that $h_\lambda$ can be lifted successively for $q$ times. Therefore by the Main Lemma, we obtain a sequence of holomorphic motions $h^{(n)}_\lambda$ of $P$ over $(\Delta_\eps,0)$, such that $h^{(0)}_\lambda=h_{\lambda}$ and $h^{(n+1)}_\lambda$ is the lift of $h^{(n)}_\lambda$
and such that $h^{(n)}_\lambda(x)\in \pm S_{\theta_n}$ for all $n$ and all $x\in P$ where $\theta_n\to 0$ as $n\to \infty$.
Thus $(f_c,f_w,\textbf{p})_W$ has the lifting property and by the Main Theorem,  
the transversality condition (\ref{fiortrans2}) holds.

Alternatively, the uniqueness of $c$ follows directly from the Main Lemma. Indeed, let
$\tilde{f}=f_{\tilde c}$ be a map with the same kneading sequence as $f_{c}$. Then one can define a real holomorphic motion $h_{\lambda}$ over some domain $\Omega\ni 0,1$ such that $h_{\lambda}(f^n(0))=\tilde{f}^n(0)$ for $\lambda=1$. As above, for $i>0$ let $h_\lambda^{(i)}$ be the lift
of $h_\lambda^{(i-1)}$.  As we have just shown,
$h_\lambda^{(i)}(c)$ is contained in the sector $-S_{\theta_n}$ with $\theta_n\to 0$,
this sequence of functions $\lambda\to h_\lambda^{(i)}(c)$ has to converge to a constant function. Since by construction
of the lifts $\tilde{c}=h_1^{(n)}(c)$ for each $n\ge 1$  we conclude that $\tilde{c}=c$.
\end{proof}
\subsection{Proof of the Main Lemma}
\begin{lemma}\label{lem:Schwarz}
For any $\theta\in (0, \pi)$ and $0<t<1$, if $z\in D_{\theta}$ then $z^t\in D_{\theta}$.
\end{lemma}
\begin{proof} This is a well-known consequence of the Schwarz lemma, due to Sullivan.
\end{proof}

When $\angle 01z$ is much smaller than $\angle 10z$, we have the following improved estimate.
\begin{lemma} \label{lem:pbtriangle}
For any $\eps>0$, there is $\delta>0$  such that the following holds.
For $z\in D_{\theta}$ with $\theta\in (0, \pi/2]$ and $\angle 01z<\delta\theta$ and for any $0<t<1$, we have $\angle 01z^t< \eps \theta.$
\end{lemma}
\begin{proof} Write $z=re^{i\alpha}$ where $r>0$ and $\alpha\in (0, \theta)$ and write $\alpha'=t\alpha$ and $\beta'=\angle 01z^t$. By assumption, $\alpha+\beta\le \theta$. By the sine theorem,
$$r=\frac{\sin \beta}{\sin (\alpha+\beta)}$$
and $$r^t=\frac{\sin \beta'}{\sin (t\alpha +\beta')}.$$

If $\alpha+\beta<\eps\theta$ then by Lemma~\ref{lem:Schwarz}, $\alpha'+\beta'\le \alpha+\beta<\eps \theta.$ Assume now $\alpha+\beta\ge \eps \theta$.
Let $K>0$ be a large constant such that
$$\frac{t}{K^t-1}<\eps \text{ for any } 0<t<1.$$
Assume $\beta<\delta\theta$ for $\delta$ small. Then $r<1/K$. Thus
$$\tan\beta'=\frac{r^t \sin t\alpha}{1-r^t\cos t \alpha} \le \frac{tr^t}{1-r^t} \alpha\le \frac{t}{K^t-1}\alpha<\eps \theta.$$
\end{proof}

\begin{lemma} \label{lem:pbangle}
Let $\varphi_\lambda$ be a $\theta$-regular motion with $\theta\in (0, \pi/10]$ and let $\psi_\lambda$ be its lift. For $x,y\in P$ so that $xy\ge 0$
let $x_\lambda=\psi_\lambda(x)$, $y_\lambda=\psi_\lambda(y)$, $u_\lambda=\varphi_\lambda(f(x))$, $v_\lambda=\varphi_\lambda(f(y))$ and $c_\lambda=\varphi_\lambda(f(0))$.

For any $\eps>0$ there is $\ell_0$ and $\delta>0$ such that if $\ell>\ell_0$ then the following hold.
\begin{enumerate}
\item If $f(x)\le 0\le f(y)$ then $\angle 0x_\lambda y_\lambda \ge \pi-\eps \theta$ for all $\lambda$.
\item Let $0<f(x)<f(y)$. Then (i) $\angle 0x_\lambda y_\lambda \ge \angle 0 u_\lambda v_\lambda -\frac{8\theta}{\ell}$. If, moreover,
$c_\lambda\in -S_{\theta_1}$ and $u_\lambda,v_\lambda\in S_{\theta_1}$ for some $\theta_1\in (0, 4\theta/\ell]$ then (ii) $x_\lambda,y_\lambda\in \pm S_{\theta_1/\ell}$ and $\angle 0x_\lambda y_\lambda \ge \angle 0 u_\lambda v_\lambda -2 \theta_1$.
\item Suppose $f(x)<f(y)<0$ and
$$\alpha=\pi-\min (\angle c_\lambda v_\lambda 0, \angle u_\lambda v_\lambda 0)<\delta \theta.$$
Then $$\angle 0 x_\lambda y_\lambda \ge \pi-\eps \theta.$$
\end{enumerate}
\end{lemma}

\begin{proof} Note that $\triangle 0 xy$ is the image of $\triangle c_\lambda u_\lambda v_\lambda$ under an appropriate branch of $z\mapsto (z-c_\lambda)^t$.  Since $\angle xoy<8\theta/\ell$, an upper bound on $\angle oyx$ implies a lower bound on $\angle oxy$.

(1) In this case, we have $u_\lambda \in -\overline{S_{4\theta/\ell}}$ and $v_\lambda \in \overline{S_{4\theta/\ell}}$, so
$$\angle 0u_\lambda v_\lambda \le  4\theta/\ell,$$
and $$\angle 0v_\lambda u_\lambda \le 4\theta/\ell.$$
In particular,
$$\angle c_\lambda u_\lambda v_\lambda \ge \angle c_\lambda u_\lambda 0 - \angle 0u_\lambda v_\lambda \ge \pi-\theta - 4\theta/\ell\ge \pi-5 \theta.$$
By Lemma~\ref{lem:pbtriangle}, the statement follows.

(2) In this case,
$$\angle c_\lambda u_\lambda v_\lambda \ge \angle 0 u_\lambda v_\lambda -\angle 0 u_\lambda c_\lambda\ge  \angle 0 u_\lambda v_\lambda - 8\theta /\ell.$$
Thus by Lemma~\ref{lem:Schwarz}, the conclusion (i) follows; (ii) is similar.

(3) In this case,
$$\angle c_\lambda u_\lambda v_\lambda \ge \angle c_\lambda u_\lambda 0-\angle v_\lambda u_\lambda 0\ge \pi-\theta-\alpha\ge \pi-2\theta$$
and $$\angle c_\lambda v_\lambda u_\lambda\le 2\pi-(\angle c_\lambda v_\lambda 0+ \angle 0v_\lambda u_\lambda) \le 2\alpha.$$
So the conclusion follows from Lemma~\ref{lem:pbtriangle}.
\end{proof}

Now suppose that we have a sequence of $\theta$-regular holomorphic motions $h_\lambda^{(i)}$ of $P$, $i=0,1,\ldots, q-1$  over the same marked domain $(\Omega,0)$, such that $h_\lambda^{(i)}$ is a lift of $h_\lambda^{(i-1)}$ for all $1\le i<q$. Then $h^{(q)}_\lambda$, lift of $h_\lambda^{(q-1)}$ is well-defined and satisfies the condition (A1) with the same constant $\theta$.
For each $0\le i\le q$, $\lambda\in\Omega$ and $x,y\in P$ so that $0<|x|<|y|$ and $xy>0$, let
\begin{align*}
& \theta_\lambda^i(x,y)=\pi-\\
&\inf\{\angle 0h_\lambda^{(i)}(z_1) h_\lambda^{(i)}(z_2): z_1, z_2\in P, 0<|z_1|\le |x|<|y|\le |z_2|, \ \ xz_1>0, xz_2>0\}\\
\ge & \pi-\angle 0h^{(i)}_\lambda(x)h_\lambda^{(i)}(y).
\end{align*}
Furthermore,
given  any $x,y\in P$, $xy>0$ (but not necessarily $|x|<|y|$), denote
$$\hat{\theta}^i(x,y)=\theta^i(x\wedge y, x\vee y)$$ where $x\wedge y=x/|x|\min(|x|,|y|)$ and $x\vee y=x/|x|\max(|x|,|y|)$.

\begin{lemma}\label{lem:pbangle1}
Consider $0\le i<q$, $x,y\in P$ where $xy>0$ and $\lambda\in\Omega$. For any $\eps>0$ there is $\delta>0$ and $\ell_0>0$ such that if $\ell\ge \ell_0$, then the following hold.
\begin{enumerate}
\item If $f(x)\le 0\le f(y)$ then $$\hat{\theta}^{i+1}_\lambda(x,y)\le \eps\theta.$$
\item Let $r\ge 1$ be such that $i+r\le q$. If $0<f^j(x)<f^j(y)$ for all $1\le j\le r$, then $$\hat{\theta}_\lambda^{i+r}(x,y)\le \hat{\theta}_\lambda^i(f^r(x),f^r(y))+\eps\theta.$$
\item If $f(x)<f(y)<0$ and $\hat{\theta}_\lambda^i(f(x), f(y))<\delta \theta$,
then $$\hat{\theta}_\lambda^{i+1}(x,y)\le 4\max (\eps\theta, \hat{\theta}_\lambda^i(f(x), f(y))).$$
\end{enumerate}
\end{lemma}
\begin{proof}
Note that $f(x)<f(y)$ implies $|x|<|y|$.

(1) For each $0<|z_1|\le |x|<|y|\le |z_2|$ as in the definition of $\theta_\lambda^i(x,y)$ we have $f(z_1)\le 0$ and $f(z_2)\ge 0$.  So by Lemma~\ref{lem:pbangle} (1), (applying to $\varphi=h^{(i)}$ and $\psi=h^{(i+1)}$), $\angle 0h_\lambda^{(i+1)}(z_1)h_\lambda^{(i+1)}(z_2)\ge \pi-\eps\theta$.  Thus the statement holds.

(2) Consider $0<|z_1|\le |x|<|y|\le |z_2|$ so that $z_1z_2>0$. Then $f(z_2)>0$. If $f(z_1)\le 0$, then by Lemma~\ref{lem:pbangle} (1),
$\angle 0h_\lambda^{(i+r)}(z_1)h_\lambda^{(i+r)}(z_2)\ge \pi-\eps\theta$.
Assume $f(z_1)>0$ and let $r_1\in \{1,\cdots,r\}$ be maximal such that $0<f^j(z_1)\le f^j(x)<f^j(y)\le f^j(z_2)$ for all $1\le j\le r_1$.
Notice that then
$$0<f^{r_1}(z_2)\le f^{r_1-1}(f^2(0))<f^{r_1-2}(f^2(0))<\cdots<f^2(0).$$
Let us show that for all $k\in \{0,\cdots,r_1-1\}$,
\begin{equation}\label{r1}
h_\lambda^{(k+i+r-r_1)}(f(0))\in -S_{4\theta/\ell^{k+1}},
\end{equation}
and
\begin{equation}\label{r2}
h_\lambda^{(k+i+r-r_1)}(f^{r_1-k}(z_1)), h_\lambda^{(k+i+r-r_1)}(f^{r_1-k}(z_2))\in S_{4\theta/\ell^{k+1}}.
\end{equation}
Indeed, this holds for $k=0$ as $h_\lambda^{(i+r-r_1)}$ is $\theta$-regular. Now, for $1\le k\le r_1-1$, (\ref{r1})-(\ref{r2}) follows by a successive application of the second part of Lemma~\ref{lem:pbangle} (2).
This proves (\ref{r1})-(\ref{r2}).
In turn, using (\ref{r1})-(\ref{r2}) and again applying successively Lemma~\ref{lem:pbangle} (2),
\begin{align*}
&\angle 0h^{(i+r)}_\lambda(z_1) h^{(i+r)}_\lambda(z_2)> \angle 0h^{(i+r-r_1)}_\lambda(f^{r_1}(z_1)) h^{(i+r-r_1)}_\lambda(f^{r_1}(z_2))-2\sum_{k=0}^\infty \frac{4\theta}{\ell^{k+1}}=\\
&\angle 0h^{(i+r-r_1)}_\lambda(f^{r_1}(z_1)) h^{(i+r-r_1)}_\lambda(f^{r_1}(z_2))-\frac{8\theta}{\ell-1}.
\end{align*}
Consider two cases. If $r_1<r$, then $f^{r_1+1}(z_1)\le 0$ and $f^{r_1+1}(z_2)>0$ and by Lemma~\ref{lem:pbangle} (1),
$$\angle 0h^{(i+r-r_1)}_\lambda(f^{r_1}(z_1)) h^{(i+r-r_1)}_\lambda(f^{r_1}(z_2))\ge \pi-{\eps}\theta$$
for any $\ell$ large enough.
If $r_1=r$,
$$\angle 0h^{(i)}_\lambda(f^{r}(z_1)) h^{(i)}_\lambda(f^{r}(z_2))\ge \pi-\theta_\lambda^i(f^r(x),f^r(y)).$$
In any case,
$$\angle 0h^{(i+r)}_\lambda(z_1) h^{(i+r)}_\lambda(z_2)>\pi-\theta_\lambda^i(f^r(x),f^r(y))-\eps\theta$$
provided $\ell$ is large enough.
Thus the statement holds.

(3) Notice that in this case $\hat{\theta}^i_\lambda(f(x),f(y))=\theta^i_\lambda(f(y),f(x))$. Consider $0<|z_1|\le |x|<|y|\le |z_2|$ so that $z_1z_2>0$. If $f(z_2)>0$ then by Lemma~\ref{lem:pbangle} (1),
$\angle 0h_\lambda^{(i+1)}(z_1)h_\lambda^{(i+1)}(z_2)\ge \pi-\eps\theta$.
Assume $f(z_2)<0$. Then $0>f(z_2)\ge f(y)>f(x)>f(z_1)>c$. So
$$\angle h^{(i)}_\lambda(c)h^{(i)}_\lambda(f(z_2))0\ge \pi-\theta_\lambda^i(f(y), f(x))$$
and $$\angle h^{(i)}_\lambda(f(z_1)) h^{(i)}_\lambda (f(z_2))0 \ge \pi-\theta_\lambda^i (f(y), f(x)).$$
By Lemma~\ref{lem:pbangle} (3),
$$\angle 0 h^{(i+1)}_\lambda  (z_1)h^{(i+1)}_\lambda (z_2)\ge \pi- 4 \max (\theta_\lambda^i(f(y), f(x)), \eps\theta),$$
provided that $\theta^i_\lambda(f(y), f(x))/\theta$ is small enough and $\ell$ is large enough.
\end{proof}

\begin{proof}[Completion of proof of the Main Lemma]
It is easy to check that $h^q_\lambda$ satisfies the condition (A1) with $S_{4\theta/\ell}$ replaced by $S_{2\theta/\ell}$. It remains to check that  for $x,y\in P$, $0<|x|<|y|$ and $xy>0$ implies $\angle 0 h_\lambda^q(x) h_{\lambda}^q(y)>\pi-\theta/2$. Since the critical point is periodic, there is a minimal integer $p$, less than the period $q$ of the critical point, such that $$f^p([x,y])\ni 0.$$
Let us define $p-1=m_0>m_1>\cdots>m_{j_0-1}>m_{j_0}=0$ inductively as follows. Given $m_i$, let $m_{j+1}\in \{0,1\cdots,m_j-1\}$ be the maximal so that $f^{m_{j+1}}([x,y])\subset \R^-$ if it exists and $m_{j+1}=0$ otherwise.
Note that $j_0\le L+1$.
Let
$$\kappa_{m_j} =\hat\theta_\lambda^{q-m_j} (f^{m_j}(x), f^{m_j}(y))/\theta, j=0,1,\ldots, j_0.$$
Fix $\eps>0$ small. Assume that $\ell$ is large.
Then by Lemma~\ref{lem:pbangle1} (1),
$$\kappa_{m_0}=\kappa_{p-1}\le \eps.$$
For each $0<j\le j_0$, by Lemma~\ref{lem:pbangle1} (2) and (3),
$$\kappa_{m_{j+1}}\le 4 \kappa_{m_j} + 4\eps$$
provided that $\kappa_{m_j}$ is small enough and $\ell$ is large enough.
Therefore, provided that $\ell$ is large enough, we have $\kappa_0<1/2$. It follows that
$$\angle 0 h^{(q)}_\lambda(x) h_\lambda^{(q)}(y) \ge \pi-\kappa_0\theta \le \pi-\theta/2.$$
\end{proof}

\section{Families without the lifting property}

In this appendix we will give a few examples of families for which the lifting property does not hold.

\subsection{Remark on the lifting property for the flat family}
Using the notations of Section~\ref{subsec:flatcritical}
let $b=2(e\ell)^{1/\ell}$, $c=-\beta=-\ell^{1/\ell}$ and $f=F_c$ so that $0\mapsto c\mapsto \beta\mapsto \beta$ by $f$. By Remark~\ref{rem:flat} the transversality fails for $(f,F,{\bf p})$. (It can be also checked directly that
the function $R(w)=F^2_w(w)-F_w(w)$ vanished at $w=c$, not identically zero but $R'(c)=0$.)
Therefore, by the Main Theorem, this triple does not have the lifting property.

The purpose of this remark is to give a more direct argument that the lifting property does  not hold  for $(f,F,{\bf p})_W$. So let
us assume by contradiction  that $(f,F,{\bf p})$ has the lifting property. Fix $r>0$ small so that the function $E(z):=\exp\{-z^{-\ell}\}$ is univalent
in a disk $B(\beta, r)\subset U^+$.
By assumption, given an arbitrary holomorphic motion $h^0_\lambda$ of $\{c,\beta\}$ there exist $\epsilon>0$ and a sequence of holomorphic motions $\{h^{(k)}_\lambda\}_{k=0}^\infty$ of $\{c,\beta\}$ over $\D_\epsilon$ such that for all $\lambda\in \D_\epsilon$ and all $k\ge 0$ we have
 that $-h^{(k)}_\lambda(c),h^{(k)}_\lambda(\beta)\in B(\beta,r)$. Therefore,
$$b E(-h^{(k+1)}_\lambda(c))+h^{(k)}_\lambda(c)=h^{(k)}_\lambda(\beta),$$
$$b E(h^{(k+1)}_\lambda(\beta))+h^{(k)}_\lambda(c)=h^{(k)}_\lambda(\beta)$$
and $h^{(k+1)}_\lambda(c)=-h^{(k+1)}_\lambda(\beta)$. Let us now choose $h^0_\lambda(c)=c-\lambda$ and $h^0_\lambda(\beta)=-h^0_\lambda(c)=\beta+\lambda$.
Then $h^{(k)}_\lambda(c)=-h^{(k)}_\lambda(\beta)$ holds for all $k\ge 0$.
Hence, for $a_k(\lambda):=h^{(k)}_\lambda(\beta)$ and all $\lambda\in \D_\epsilon$, $k\ge 0$,
$$b\exp\{-\frac{1}{a_{k+1}(\lambda)^\ell}\}=2a_k(\lambda).$$
On the other hand, it is easy to check that the for function $f:(0,+\infty)\to(0,+\infty)$,
$$f(x)=\frac{b}{2}\exp\{-\frac{1}{x^\ell}\},$$
we have: $f(\beta)=\beta$, $Df(\beta)=1$, $D^2f(\beta)<0$, $f: [\beta,+\infty)\to [\beta,b/2)$ is an increasing homeomorphism so that $f^k(x)\to \beta$ as $k\to +\infty$ for all $x\in [\beta,+\infty)$.
Let $U: [\beta,b/2)\to [\beta,+\infty)$ be a branch of $f^{-1}$ such that $U(\beta)=\beta$. Since $a_k(0)=\beta$ and functions $a_k$ are continuous in $\D_\epsilon$,
it follows that
$$a_{k}(\lambda)=U^k(a_0(\lambda))$$
for all $k>0$ and all $\lambda$ provided $a_0(\lambda)=\beta+\lambda\ge \beta$.
Fix $\lambda_0\in (0,\epsilon)$ so that $a_0(\lambda_0)=\beta+\lambda_0>\beta$.
It follows that $h^{(k)}_{\lambda_0}(\beta)=a_{k}(\lambda_0)=U^k(a_0(\lambda_0))\to +\infty$ as $k\to \infty$, a contradiction with the definition of the lifting property.
\subsection{Spectrum of the transfer operator and linear coordinate changes of the quadratic family}
Consider the standard holomorphic deformation of a critically finite quadratic map $(g,G,{\bf p})$,
that is, ${\bf p}(w)=0$, $G_w(z)=z^2+w$ and $g=G_{c_1}$ is so that $0$ is periodic for $g$ of period $q\ge 2$.

Given a function $\nu$ which is locally holomorphic in a neighborhood of a point $v_1\in \C\setminus \{0\}$ so that $\nu(v_1)=c_1$, let
$w=\nu(v)$ and $\varphi(w)=\varphi(\nu(v))=\nu(v)/v$.
Define
$$G^\nu_v(z)=\frac{1}{\varphi(w)}G_w(\varphi(w)z)=\frac{\nu(v)}{v}z^2+v$$
and $g_\nu=G^\nu_{v_1}$.


Then $(g_\nu, G^\nu, {\bf p})$ is a local holomorphic deformation of $g_\nu$. Denote by $\mathcal A$ respectively $\mathcal A_\nu$
the transfer operator of $(g,G,{\bf p})$ respectively of $(g_\nu,G^\nu,{\bf p})$.
Note that $1/2$ is always in the spectrum of $\mathcal A$ (see \cite{LSY}).
\begin{prop}\label{B2}
$$\det(1-\rho\mathcal A_\nu)=\frac{1-\frac{1}{2}\frac{v_1 \nu'(v_1)}{c_1}}{1-\frac{\rho}{2}}\det(1-\rho\mathcal A).$$
\end{prop}
\begin{proof}
Let $L_\nu(z)=\partial G^\nu_v(z)/\partial v|_{v=v_1}$, $v_n=g_\nu^n(0)$,
$$D_\nu(\rho)=1+\sum_{n=1}^{q-1}\rho^n\frac{L_\nu(v_n)}{(g_\nu^n)'(v_1)},$$
$$D(\rho)=1+\sum_{n=1}^{q-1}\rho^n\frac{1}{(g^n)'(c_1)}.$$
We have to varify the following identity:
$$D_\nu(\rho)=\frac{1-\rho(1-\frac{1}{2}\frac{v_1 \nu'(v_1)}{\nu(v_1)})}{1-\frac{\rho}{2}}D(\rho).$$
Let us indicate its proof. We have:
$(g_\nu^n)'(v_1)=(g^n)'(c_1)$, $v_n=c_n/\varphi(c_1)$ (in particular, $\varphi(c_1)=c_1/v_1$, and
$$L_\nu(v_n)=(\frac{\nu(v)}{v})'|_{v=v_1}v_n^2+1.$$
Then the above identity turns out to be equivalent to the following one:
$$2(1-\frac{\rho}{2})\frac{1}{c_1}\sum_{n=1}^{q-1}\rho^{n-1}\frac{c_n^2}{(g^n)'(c_1)}=D(\rho)$$
which is checked directly using  $c_{n+1}-c_n^2=c_1$ for $1\le n\le q-2$ and $-c_{q-1}^2=c_1$.
\end{proof}
\begin{coro}
The transversality of $(g_\nu,G^\nu,{\bf p})$ fails if and only if $\nu'(v_1)=0$.
\end{coro}
\begin{coro}
Let $c_1$ and $\nu$ be real. Then $(g_{\nu},G^{\nu},{\bf p})$ has the lifting property if and only if the positive transversality property holds.
\end{coro}
\begin{proof}
By the Main Theorem, the lifting property yields the positive transversality. Conversaly, by Proposition~\ref{B2}, $D_\nu(1)>0$ implies that the spectrum of $\mathcal A_\nu$ belongs to the open unit disk which in turn implies the lifting property.
\end{proof}

\section{The lifting property in the setting of rational mappings} \label{sec:liftingclassical}
The goal of this section is to show how to apply the method developed in this paper to deal with transversality problems in the classical case of polynomials and rational maps.
It is natural to assume that the maps involved are suitably normalized, so we shall only consider the following situations:

(a) $f$ is a map in $\textbf{P}_d$, the family of monic centered polynomials of degree $d\ge 2$;

(b) $Z$ is a set with $\#Z=3$ and $\textbf{Rat}_d^Z$ is the family of all rational maps $f$ of degree $d$ such that $f(Z)=Z$ and such that $Z$ is disjoint from the critical orbit of $f$.

Note that for each rational map $f$ of degree $d\ge 2$, it is possible to find $Z$, consisting of either a cycle of period $3$, or a fixed point and a cycle of period $2$, such that $f\in \textbf{Rat}_d^Z$. Let $\mathcal{U}=\textbf{P}_d$ in case (a) and $\mathcal{U}=\textbf{Rat}_d^Z$ in case (b).

In case (b), we assume without loss of generality that critical points and their orbits avoids the point at $\infty$. Let $c_1, c_2, \cdots, c_\nu$ be the distinct (finite) critical points of $f$ with multiplicities $\mu_1, \mu_2, \cdots, \mu_\nu$ and let $v_j=f(c_j)$.
In the following, we fix $f$ and let $\mathcal{U}^f$ denote the subcollection of maps in $\mathcal{U}$ which have exactly $\nu$ critical points with multiplicity $\mu_1,\mu_2,\ldots,\mu_\nu$. It is well-known that $\mathcal{U}^f$ is a complex manifold of dimension $\nu$ and the critical values are holomorphic coordinates.
See for example \cite{LSvS}.

\begin{prop}\label{sect3p}
There is a neighbourhood $W$ of $(v_1, v_2,\cdots, v_\nu)$ in $\C^\nu$ such that
$W\ni w \mapsto f_w$ in $\textbf{P}_d$ (resp. $\textbf{Rat}_d^Z$) is biholomorphsim from $W$ to a neighborhood of $f$ in $\mathcal{U}^f$, and a holomorhic function $\textbf{p}:W\to \C^\nu$, such that
$p_j(w)$ is a critical point of $f_w$ of multiplicity $\mu_j$  
and
$$w=(f_{w}(p_1(w)),f_{w}(p_2(w)),\cdots,f_{w}(p_\nu(w))).$$
\end{prop}

Now we assume that $f$ is critically finite. Let $P=P(f)$, $P_0=\{c_1, c_2,\ldots, c_\nu\}$ and let $U$ be a small neighborhood of $P\setminus P_0$. Then $f$, restricting to $P_0\cup U$, is a marked map in the sense of \S\ref{subsec:markedmap}. We shall use the same notation $f$ for the marked map. Moreover, defining $F(w, z)=f_w(z)$, $(f, F, \textbf{p})_W$ is a holomorphic deformation of $f$ in the sense of \S\ref{subsec:holdefor}.

\begin{theorem}\label{thm:ratlift} The holomorphic deformation $(f,F,{\bf p})_W$ satisfies the lifting property.
\end{theorem}

\begin{proof} Let $h_\lambda^{(0)}$ be an arbitrary holomorphic motion of $f(P)$ over $(\D,0)$. By Bers-Royden~\cite{BersRoyden}, there exists $\eps>0$ such that
$h_\lambda^{(0)}$, $\lambda\in \D_\eps$, extends to a holomorphic motion of $\C$ over $(\D_\eps, 0)$ which satisfies the following: in case (a), for $|z|$ large enough, $h_\lambda^{(0)}(z)$ is holomorphic in $z$ and $h_\lambda^{(0)}(z)=z+o(1)$ near infinity, moreover,
fix a big enough disk $V$ such that $f^{-1}(V)\subset V$ and such that the complex dilatation $\mu_\lambda^{(0)}$ of $h_\lambda^{(0)}$ is supported in $V$ for all $\lambda\in \D_\eps$,
and in case (b), $h_\lambda(z)=z$ for all $z\in Z$. Moreover, the complex dilatation $\mu_\lambda^{(0)}$ of $h_\lambda^{(0)}$ depends holomorphically in $\lambda$.
Define $\mu_\lambda^{(k)}=(f^k)^* (\mu_\lambda^{(0)})$ for each $k\ge 1$ (here $f$ is considered as a globally defined map) and let $h_\lambda^{(k)}$ denote the unique qc map with complex dilatation $\mu_\lambda^{(k)}$ and such that the following holds: in case (a), $h_\lambda^{(k)}(z)=z+o(1)$ near infinity; and in case (b), $h_\lambda^{(k)}(z)=z$ for all $z\in Z$. Then by the Measurable Riemann Mapping Theorem, $h_\lambda^{(k)}$ is a holomorphic motion of $\C$ over $(\D_\eps, 0)$. Let us show that for each $k\ge 0$, $h_\lambda^{(k+1)}$, restricting to $f(P)$, is a lift of $h_\lambda^{(k)}$, restricting to $f(P)$, with respect to $(f, F, \textbf{p})_W$. This amounts to prove the following:

\noindent
{\bf Claim.} For $|\lambda|$ small enough, we have $$h_\lambda^{(k)}\circ f\circ (h_\lambda^{(k+1)})^{-1}=f_{(h_\lambda^{(k)}(v_1), h_{\lambda}^{(k)}(v_2),\cdots, h_\lambda^{(k)}(v_\nu))}.$$ 

\noindent
{\bf Proof of Claim:}  for each $\lambda\in \D_\eps$,  the complex dilatation of $h_\lambda^{(k+1)}$ is  the $f$-pull back of that of $h_\lambda^{(k)}$, and therefore the function
$g_\lambda:=h_\lambda^{(k)}\circ f\circ (h_\lambda^{(k+1)})^{-1}$ is holomorphic in $\C$. It is a branched covering of degree $d$, so it is either a polynomial or a rational function of degree $d$. By the normalization of both $h_\lambda^{(j)}$, $j=k, k+1$, in case (a), $g_\lambda\in\mathcal{U}$, and hence $g\in\mathcal{U}^f$. Clearly the critical values of $g_\lambda$ are $h_\lambda^{(k)}(v_i)$. The claim follows.

To complete the proof, notice that in the case (a), $\mu_\lambda^{(k)}$ are supported in $f^{-k}(V)\subset V$ and by compactness of K-qc maps the conclusion follows.
\end{proof}

\begin{coro} Under the circumstances above, either $(f, F, \textbf{p})_W$ satisfies the transversality property or
or $f$ is a flexible  Latt\'es map.
\end{coro}
\begin{proof} Suppose that $(f, F, \textbf{p})_W$ does not satisfies the transversality property. Then by the Main Theorem, there is a non-trivial local holomorphic family $f_t$ passing through $f$
of critically finite rational maps. By McMullen-Sullivan \cite{McS}, $f$ carries an invariant line field in its Julia set.  Since the postcritical set of $f$ is finite,  $f$ must be a flexible Latt\'es map,
see e.g. Corollary 3.18 of \cite{McM}.
\end{proof}

\bibliographystyle{plain}             

\begin{thebibliography}{1234}
\bibitem{Ast} M. Astorg, {\em Summability condition and rigidity for finite type maps}, ArXiv \url{http://arxiv.org/abs/1602.05172}.
\bibitem{ALdM} A. Avila, M. Lyubich and W.  de Melo, {\em Regular or stochastic dynamics in real analytic families of unimodal maps},
Inventiones mathematicae {\bf 154} (2003), 451--550.
\bibitem{BersRoyden} L. Bers, H.L. Royden {\em Holomorphic families of injections}, Acta Math. {\bf 157} (1986), 259-286.
\bibitem{Br} H. Bruin, {\em Non-monotonicity of entropy of interval maps}, Physics Letters A, 1995, 4--8.
\bibitem{BS} H.\ Bruin, S.\ van Strien,
{\em Monotonicity of entropy for real multimodal maps,} Journal of the AMS. 28 (2015), 1--61.
\bibitem{BE} X. Buff and A. Epstein,  {\em Bifurcation measure and postcritically finite rational maps},  Complex dynamics, 491--512, A K Peters, Wellesley, MA, 2009.
\bibitem{Betal} X. Buff, A. Epstein, S. Koch, {\em Eigenvalues of the Thurston operator},  ArXiv:1707.01990, 2017.
\bibitem{BranHub} B. Branner and J.H.  Hubbard, {\em The iteration of cubic polynomials. I. The global topology of parameter space},
 Acta Math. {\bf 160} (1988), no. 3-4, 143--206.
\bibitem{Chirka} E.M. Chirka, {\em Complex Analytic Sets}, Kluwer, 1989.
\bibitem{CvS} T. Clark and S. van Strien, {\em Quasi-symmetric rigidity in one-dimensional dynamics}. ArXiv \url{https://arxiv.org/abs/1805.09284}. 
\bibitem{D}
A. Douady.
\newblock {\em Topological entropy of unimodal maps: monotonicity for quadratic polynomials.} {\em Real and complex dynamical systems}, 65--87
\newblock NATO Adv. Sci. Inst. Ser. C Math. Phys. Sci., 464, Kluwer Acad. Publ., Dordrecht, 1995.
\bibitem{DH1}  A. Douady and J.H. Hubbard,  {\'E}tude dynamique des polyn\^omes complexes. Partie I. and II.
 Publications Math\'ematiques d'Orsay. Universit\'e de Paris-Sud, D\'epartement de Math\'ematiques, Orsay, 1984-1985.
 Translation into English: Exploring the Mandelbrot set, see \url{http://pi.math.cornell.edu/~hubbard/OrsayEnglish.pdf}
 \bibitem{DH2} A. Douady and J.H. Hubbard,
A proof of Thurston's topological characterization of rational functions.  Acta Math. {\bf 171} (1993), no. 2, 263--297.
\bibitem{Ep1} A. Epstein. {\em Towers of finite type complex analytic maps},  PhD thesis, City University of New York, 1993.
\bibitem{Ep3} A. Epstein. {\em Infinitesimal Thurston rigidity and the Fatou-Shishikura inequality}. Stony Brook IMS Preprint 1999/1, 1999.
\bibitem{Ep2}  A. Epstein, {\em Transversality in holomorphic dynamics}, \url{http://homepages.warwick.ac.uk/~mases/Transversality.pdf} and slides of talks available in
\url{https://icerm.brown.edu/materials/Slides/sp-s12-w1/Transversality_Principles_in_Holomorphic_Dynamics_\%5D_Adam_Epstein,_University_of_Warwick.pdf}
\bibitem{Er} A. Eremenko, Private communication.
\bibitem{ELS} A. Eremenko, G. Levin and M. Sodin, {\em On the distribution of zeros of a Ruelle zeta-function.} Comm. Math. Phys. 159 (1994), no. 3, 433--441.
\bibitem{GO} A.A. Goldberg and I.V. Ostrovskii, {\em Value distribution of meromorphic functions.}
Translated from the 1970 Russian original by Mikhail Ostrovskii. With an appendix by Alexandre Eremenko and James K. Langley. Translations of Mathematical Monographs, 236. American Mathematical Society, Providence, RI, 2008. xvi+488 pp.
\bibitem{Goluzin} G.M.  Goluzin, {\em Geometric theory of functions of a complex variable,} Translations of Mathematical Monographs, Vol. 26 American Mathematical Society, Providence, R.I. 1969.
\bibitem{Ko} S.F.  Kolyada, {\em One-parameter families of mappings of an interval with a negative Schwartzian derivative which violate the monotonicity of bifurcations}.
Ukrainian Math. J. {\bf 41} (1989), no. 2, 230--233.
\bibitem{LSvS} G. Levin, W. Shen and S. van Strien, {\em Transversality for critical relations of families rational maps: an elementary proof}.
 In: Pacifico M., Guarino P. (eds) New Trends in One-Dimensional Dynamics. Springer Proceedings in Mathematics \& Statistics, vol 285. Springer, Cham. pp 201--220.
\bibitem{LSvS2} G. Levin, W. Shen and S. van Strien, {\em Monotonicity of entropy and positively oriented transversality for families of interval maps}. https://arxiv.org/abs/1611.10056
\bibitem{LSvScompanion} G. Levin, W. Shen and S. van Strien, {\em In preparation}.
\bibitem{LSY0} G.  Levin,  M. Sodin and P. Yuditski, {\em A Ruelle operators for a real Julia set.} Commun. Math. Physics 141 (1991), 119--132
\bibitem{LSY} G.  Levin,  M. Sodin and P. Yuditskii, {\em Ruelle operators with rational weights for Julia sets.} J. Anal. Math. 63 (1994), 303--331
\bibitem{Le0} G. Levin, {\em On Mayer's conjecture and zeros of entire functions.} Ergodic Theory Dynam. Systems 14 (1994), no. 3, 565--574.
\bibitem{Le1} G. Levin, On an analytic approach to the Fatou conjecture. Fund. Math. 171 (2002), no. 2, 177--196.
\bibitem{Le2} G. Levin, {\em Multipliers of periodic orbits in spaces of rational maps}.  Ergodic Theory Dynam. Systems 31 (2011), no. 1, 197--243.
\bibitem{Le3} G. Levin,  Perturbations of weakly expanding critical orbits. {\em Frontiers in complex dynamic}s, 163--196, Princeton Math. Ser., 51, Princeton Univ. Press, Princeton, NJ, 2014.
\bibitem{Le4} G. M. Levin, {\em Polynomial Julia sets and Pad\'es approximations}. Proceedings of XIII Workshop on Operators Theory in Functional Spaces (Kyubishev, October 6-13, 1988). Kyubishev State University, Kyubishev, 1988, pp. 113-114 (in Russian)
\bibitem{Mak} P. Makienko, Remarks on the Ruelle operator and the invariant line fields problem. II.
Ergodic Theory Dynam. Systems 25 (2005), no. 5, 1561--1581.
\bibitem{McM} C. McMullen, {\em  Complex dynamics and renormalisation}. Princeton University Press, 1994, Princeton.
\bibitem{McS} C. McMullen and D. Sullivan. {\em Quasiconformal homeomorphisms and dynamics. III. The Teichm\"uller space of a holomorphic dynamical system},
Adv. Math. {\bf 135} (1998), no. 2, 351--395.
\bibitem{Mil} J. Milnor, {\em Thurston's algorithm without critical finiteness}, in Stony Brook IMS Preprint 1992/7, see also arXiv:math/9205209v1.
\bibitem{MS} M. Misiurewicz and W. Szlenk, {\em Entropy of piecewise monotone mappings}. Studia Math. {\bf 67} (1980), no. 1, 45--63.
\bibitem{MT} J. Milnor and W. Thurston.
\newblock {\em On iterated maps of the interval}. {\em Dynamical systems (College Park, MD, 1986--87)},  465--563,
\newblock Lecture Notes in Math., 1342, Springer, Berlin, 1988.
\bibitem{MTr} J.\ Milnor, C.\ Tresser,
{\em On entropy and monotonicity for real cubic maps}, Comm.\ Math.\ Phys.\ {\bf 209} (2000), 123--178.
\bibitem{NY}  H.E. Nusse and J.A,  Yorke, {\em Period halving for $x_{n+1}=MF(x_n)$ where F has negative Schwarzian derivative},  Phys. Lett. A {\bf 127} (1988), no. 6-7, 328--334.
\bibitem{Radu} A.\ Radulescu,
{\em The connected isentropes conjecture in a space of quartic
  polynomials,}
Discrete Contin.\ Dyn.\ Syst.\ {\bf 19} (2007), 139--175.
\bibitem{RS} L. Rempe-Gillen and S. van Strien,  {\em Density of hyperbolicity for classes of real transcendental entire functions and circle maps}. Duke Math. J. 164 (2015), no. 6, 1079--1137.
    \bibitem{R} L. Rempe-Gillen, {\em Rigidity of escaping dynamics for transcendental entire functions}, Acta Math. 203 (2009), 235-267.
\bibitem{Rudin} W. Rudin, {\em Function Theory in the Unit Ball of $\C^n$}, Springer-Verlag, 1980.
 \bibitem{Shabat} B. V. Shabat, {\em Introduction to Complex Analysis. Part II: Functions of several variables}, AMS, v. 110, 1992.
 \bibitem{Str} S. van Strien,  {\em Misiurewicz maps unfold generically (even if they are critically non-finite.}
 Fund. Math. 163 (2000), no. 1, 39--54.
\bibitem{Str2} S. van Strien, {\em  One-Parameter Families of Smooth Interval Maps: Density of Hyperbolicity and Robust Chaos}, Proc. A.M.S. 138 (2010), 4443-4446.
\bibitem{Su}
D. Sullivan. {\em Unpublished.}
\bibitem{Tsu0} M.  Tsujii, A note on Milnor and Thurston's monotonicity theorem. Geometry and analysis in dynamical systems (Kyoto, 1993), 60--62, Adv. Ser. Dynam. Systems, 14, World Sci. Publ., River Edge, NJ, 1994.
\bibitem{Tsu1}
M. Tsujii.
\newblock {\em A simple proof for monotonicity of entropy in the quadratic family},  Ergodic Theory Dynam. Systems 20 (2000), 925--933.
\bibitem{Zd} A. Zdunik, {\em Entropy of transformations of the unit interval}, Fund. Math. 124 (1984),
235--241.
\end{thebibliography}

\vskip 1cm 

\email{genady.levin@mail.huji.ac.il, wxshen@fudan.edu.cn, s.van-strien@imperial.ac.uk} 

\bigskip 

\address{Hebrew University, Fudan University, Imperial College London} 

\end{document}